\documentclass{amsart}
\usepackage{mathrsfs}
\usepackage[all]{xy}

\usepackage{wasysym}
\usepackage{amssymb}

\usepackage{ifpdf}
\ifpdf 
  \usepackage[pdftex]{graphicx}
  \DeclareGraphicsExtensions{.pdf,.png,.jpg,.jpeg,.mps}
  \usepackage{pgf}
\else 
  \usepackage{graphicx}
  \DeclareGraphicsExtensions{.eps,.bmp}
  \usepackage{pgf}
  \usepackage{pstricks}
\fi
\usepackage{epic,bez123}
\usepackage{wrapfig}

\newtheorem{thm}{Theorem}[section]
\newtheorem{prop}[thm]{Proposition}
\newtheorem{cor}[thm]{Corollary}
\newtheorem{lem}[thm]{Lemma}

\newtheorem*{claim}{Claim}

\newtheorem*{conv}{Convention}

\theoremstyle{definition}
\newtheorem{defn}[thm]{Definition}

\theoremstyle{remark}

\newtheorem{rem}[thm]{Remark}
\newtheorem{example}[thm]{Example}

\newcommand{\G }{\mathscr{G} (G, S\cup \mathcal P)}

\newcommand{\Gx }{\mathscr{G} (G, S)}

\newcommand{\dxp }{d_{S\cup{\mathcal P}}}

\newcommand{\dx }{d_S}

\newcommand{\diam }[1]{{\|#1\|}}

\newcommand{\proj }{{{\Pi}}}

\newcommand{\lab }{\textbf{Lab}}

\newcommand{\len }{\ell}

\newcommand{\card}[1]{\lvert#1\rvert}

\usepackage[bookmarks=true, pdfauthor={YANG Wenyuan}]{hyperref}

\begin{document}

\title{Combination of fully quasiconvex subgroups and its applications}

\author{Wen-yuan Yang}

\address{Le Département de Mathématiques de la Faculté des
Sciences d'Orsay, Université Paris-Sud 11, France}
\email{yabziz@gmail.com}
\thanks{This research is supported by the ERC starting grant GA 257110 ”RaWG”.}



\date{Dec 18, 2012}

\dedicatory{}


\begin{abstract}
In this paper, we state two combination theorems for relatively
quasiconvex subgroups in a relatively hyperbolic group. Applications
are given to the separability of double cosets of certain relatively
quasiconvex subgroups and the existence of closed surface subgroups
in relatively hyperbolic groups.
\end{abstract}

\maketitle

\section{Introduction}
The combination theorems for relatively hyperbolic groups have been
developed by many authors, and possessed wide applications in the
theory of relatively hyperbolic groups. See \cite{BF}, \cite{Dah},
\cite{Ali}, \cite{MR} and \cite{Gau}, just to name a few. By
contrast, the problems for combining relatively quasiconvex
subgroups in a relatively hyperbolic group are less well-studied.
This kind of combination results also afford many important
applications in constructing good subgroups. For instance, the idea
of combining compact surfaces with parallel boundaries was initiated
in Freedman-Freedman \cite{FF} and further explored by Cooper, Long
and Reid(\cite{CLR}, \cite{CL}) to obtain closed surfaces in cusped
hyperbolic 3-manifolds.

In a relatively hyperbolic group $G$, one can ask the following
questions for the combination of relatively quasiconvex subgroups
$H, K$:

\begin{enumerate}
\item
Under what conditions the amalgamation $H \star_C K$ over $C = H
\cap K$ is embedded in $G$ as a relatively quasiconvex subgroup?

\item
Under what conditions the HNN extension $H \star_{Q_1 \sim Q_2}$
over isomorphic subgroups $Q_1, Q_2$ is embedded in $G$ as a
relatively quasiconvex subgroup?
\end{enumerate}

Recall that a relatively quasiconvex subgroup is itself a relatively
hyperbolic group. Hence whether the combination of two relatively
quasiconvex subgroups is relatively hyperbolic and also embedded
into the ambient group $G$ give two theoretical obstructions for
solving the above problems.

In hyperbolic groups, Gitik showed that these two obstructions can
be virtually eliminated in virtue of a separability property: two
quasiconvex subgroups, if their intersection is separable,  contain
(many) finite index subgroups to generate a quasiconvex amalgamation
\cite{Gitik}. Along this line, Martinez-Pedroza proved a combination
theorem in a relatively hyperbolic group for combining relatively
quasiconvex subgroups over a parabolic subgroup \cite{MarPed2}. In
the other direction,  Baker-Cooper showed that a pair of
geometrically finite subgroups with compatible parabolic subgroups
can be virtually amalgamated \cite{BC}.

In the present paper, we shall show two combination theorems in the
same spirit for relatively quasiconvex subgroups with fully
quasiconvex subgroups. Some applications of our combination results
are given to the separability of double cosets and the existence of
closed surface group. We now start by stating the combination
theorems.

\paragraph{1. \textbf{Combination theorems}}
Let $G$ be a finitely generated group hyperbolic relative to a
collection of subgroups $\mathbb P$. A \textit{fully quasiconvex
subgroup} $H$ is relatively quasiconvex in $G$ such that $P^g \cap
H$ is either finite or of finite index in $P^g$ for each $g \in G, P
\in \mathbb P$. Fully quasiconvex subgroups generalize quasiconvex
subgroups in hyperbolic groups and are receiving a great deal of
attention in the study of relatively quasiconvex subgroups, see
\cite{ManMar}, \cite{CDW} and \cite{Wise} and \cite{GePo4}.

Our first result is to deal with the (virtual) amalgamation of a
relatively quasiconvex subgroup with a fully quasiconvex subgroup,
generalizing results of Gitik in the hyperbolic case \cite{Gitik}.

\begin{thm}[Virtual amalgamation] \label{freeproduct}
Suppose $H$ is relatively quasiconvex and $K$ fully quasiconvex in a
relatively hyperbolic group $G$. Then there exists a constant
$D=D(H,K) > 0$ such that the following statements are true.

\begin{enumerate}
\item
Let $\dot H \subset H$ and $\dot K \subset K$ be such that $\dot H
\cap \dot K = C$ and $d(1, g) > D$ for any $g \in \dot H \cup \dot K
\setminus C$, where $C = H \cap K$. Then $\langle \dot H, \dot
K\rangle= \dot H \star_C \dot K$.

\item
If $\dot H, \dot K$ are, in addition, relatively quasiconvex, then
$\langle \dot H, \dot K\rangle$ is relatively quasiconvex.
\item
Moreover, every parabolic subgroup in $\langle \dot H, \dot
K\rangle$ is conjugated into either $\dot H$ or $\dot K$.

\end{enumerate}
\end{thm}

\begin{rem}
Note that maximal parabolic subgroups are fully quasiconvex. In the
case that $K$ is a maximal parabolic subgroup, it suffices to assume
$d(1, g) >D$ for any $g \in \dot K \setminus C$ in Theorem
\ref{freeproduct}. This generalizes a result in \cite{MarPed2}.
\end{rem}

We also obtain a combination theorem for glueing two parabolic
subgroups of a relatively quasiconvex subgroup such that the HNN
extension is relatively quasiconvex. Let $\Gamma^g = g\Gamma g^{-1}$
be a conjugate of a subgroup $\Gamma$ in $G$.
\begin{thm} [HNN extension] \label{hnn}
Suppose $H$ is relatively quasiconvex in a relatively hyperbolic
group $G$. Let $P \in \mathbb P$ and $f \in G$ such that $Q = P \cap
H, Q' = Q^f$ are non-conjugate maximal parabolic subgroups in $H$.
Then there exists a constant $D=D(H, P, f) > 0$ such that the
following statements are true.

\begin{enumerate}
\item
Suppose there exists $c \in P$ such that $Q^c=Q$ and $d(1, c) > D$
for any $g \in cQ$. Let $t=fc$. Then $\langle H, t\rangle =
H\star_{Q^t = Q'}$ is relatively quasiconvex.

\item
Moreover, every parabolic subgroup in $\langle H, t\rangle$ is
conjugated into $H$.
\end{enumerate}
\end{thm}
\begin{rem}
The sufficiently long element $c$ exists when $Q$ is normal and of
infinite index in $P$. In particular, this holds for the groups
hyperbolic relative to abelian groups.
\end{rem}

In the setting of hyperbolic 3-manifolds, Theorem \ref{hnn}
generalizes a theorem of Baker-Cooper \cite[Theorem 8.8]{BC}, which
was used to glue parallel boundary components of immersed surfaces
in a 3-manifold to construct closed surfaces in \cite{CL}. As an
application of our theorem, we consider the existence of closed
surface groups in a relatively hyperbolic group.

Let $H$ be the fundamental group of a compact surface $S$. Recall
that $H$ has \textit{no accidental parabolics} in a relatively
hyperbolic group $(G, \mathbb P)$ if the conjugacy
 class of elements in $H$ representing boundary components in $S$ is exactly
the elements in $H$ which can conjugated into some $P \in \mathbb
P$.
\begin{cor} \label{surface}
Suppose that $G$ is hyperbolic relative to abelian subgroups of rank
at least two. Let $H$ be the fundamental group of a compact surface
with boundary such that $H$ has no accidental parabolics in $G$.
Then there exists a closed surface subgroup in $G$ which is
relatively quasiconvex.
\end{cor}

\paragraph{2. \textbf{Our approach: admissible paths}}
The approach in proving our combination results is based on a notion
of \textit{admissible paths} in a geodesic metric space with a
system of contracting subsets. A \textit{contracting} subset is
defined with respect to a preferred class of quasigeodesics such
that any of them far from the contracting subset has a uniform
bounded projection to it. See precise definitions in Section 2.

The notion of contracting subsets turns out to compass many
interesting examples. For instance, quasigeodesics and quasiconvex
subspaces in hyperbolic spaces, parabolic cosets in relatively
hyperbolic groups \cite{GePo4}, contracting segments in CAT(0)
spaces \cite{BF2} and the subgroup generated by a hyperbolic element
in groups with nontrivial Floyd boundary \cite{GePo4}. Relevant to
our context, it is worthwhile to point out that fully quasiconvex
subgroups are contracting, as shown in \cite[Proposition
8.2.4]{GePo4}.

In terms of contracting subsets, an admissible path can be roughly
thought as a concatenation of quasigeodesics which travels
alternatively near contracting subsets and leave them in an
orthogonal way (see Definition \ref{AdmDef}). The informal version
of our main result about admissible paths is the following.
\begin{prop}[cf. Proposition \ref{admissible}]\label{admissible2}
Long admissible paths are quasigeodesics.
\end{prop}
\begin{example}
\begin{enumerate}
\item
Theorems \ref{freeproduct} and \ref{hnn} are proved by constructing
admissible paths for each element in $H \star_C K$ and $H\star_{Q^t
= Q'}$.
\item
Note that local quasigeodesics in hyperbolic spaces are admissible
paths and hence quasigeodesics (\cite{GH}, \cite{Gro}).
\item
Since contracting segments in the sense of Bestvina-Fujiwara are
contracting in our sense, Proposition \ref{admissible2} can be also
thought as a unified version of Proposition 4.2 and Lemma 5.10 in
\cite{BF2}.
\end{enumerate}
\end{example}

\paragraph{3. \textbf{Separability of double cosets}}
Recall that a subset $X$ of a group $G$ is \textit{separable} if for
any $g \in G \setminus X$, there exists a homomorphism $\phi$ of $G$
to a finite group such that $\phi(g) \notin \phi(X)$; in other
words, $X$ is a \textit{closed} subset in $G$ with respect to the
profinite topology. A group $G$ is called \textit{LERF} if every
finitely generated subgroup are separable. A \textit{slender} group
contains only finitely generated subgroups.

An application of our combination theorem, generalizing Gitik
\cite{Gitik2} and Minasyan \cite{Min}, is to give a criterion of
separability of double cosets of certain relatively quasiconvex
subgroups.
\begin{thm} \label{doublecoset}
Suppose $G$ is hyperbolic relative to slender LERF groups and
separable on fully quasiconvex subgroups. Let $H$ be relatively
quasiconvex and $K$ fully quasiconvex in $G$. If $H' \subset H,
K'\subset K$ are relatively quasiconvex in $G$ such that $H' \cap
K'$ is of finite index in $H \cap K$, then $H'K'$ is separable.
\end{thm}
\begin{rem}
If $H$ is also fully quasiconvex, then the condition on each
parabolic subgroup being LERF and slender would not be necessary.
\end{rem}

An interesting corollary is obtained as follows when each maximal
parabolic subgroup are virtually abelian. Note that abelian groups
are LERF and slender.
\begin{cor} \label{dblparacoset}
Suppose $G$ is hyperbolic relative to virtually abelian groups and
separable on fully quasiconvex subgroups. Then the double coset of
any two parabolic subgroups is separable. In particular, the double
coset of any two cyclic subgroups is separable.
\end{cor}

\begin{rem}
When $G$ is the fundamental group of a cusped hyperbolic 3-manifold
of finite volume,  Hamilton-Wilton-Zalesskii showed in \cite{HWZ}
that,  without additional assumptions, the double coset of any two
parabolic subgroups is separable. In \cite{Wise}, Wise proved that
$G$ is virtually special and thus separable on fully quasiconvex
subgroups. Hence, this corollary with Wise's result gives another
proof of their result.
\end{rem}

This paper is structured as follows. In Section 2, we give a general
development of the notion of admissible paths, which underlies the
proofs of Theorems \ref{freeproduct} and \ref{hnn}. Sections 3 \& 4
and Section 5 are devoted to prove Theorems \ref{freeproduct} and
\ref{hnn} respectively. In the final section, we give the proof of
Theorem \ref{doublecoset} and its corollary.

After the completion of this paper, the author noticed that Eduardo
Martinez-Pedroza and Alessandro Sisto proved a more general
combination theorem in \cite{MarSisto}. Our Theorem
\ref{freeproduct} is a special case of their result. However, our
methods are different and Theorem \ref{hnn} does not follow from
their result.

\section{Axiomatization: Admissible Paths}

The purpose of this section is two-fold. First, a notion of an
admissible path is introduced as a model in proving our combination
theorems in next sections. In fact, this notion arises an attempt to
unify the proofs of combination theorems.

Secondly, we pay much attention to axiomatize the discussion, with
the aim extracting the hyperbolic-like feature naturally occurred in
various contexts. The motivating examples we have in mind are
parabolic cosets in relatively hyperbolic groups and contracting
segments in CAT(0) spaces.

\subsection{Notations and Conventions}
Let $(Y, d)$ be a geodesic metric space. Given a subset $X$ and a
number $U \ge 0$, let $N_U(X) = \{y \in Y: d(y, X) \le U \}$ be the
closed neighborhood of $X$ with radius $U$. Denote by $\diam {X}$
the diameter of $X$ with respect to $d$.

Fix a (sufficiently small) number $\delta > 0$ that won't change in
the rest of paper. Given a point $y \in Y$ and subset $X \subset Y$,
let $\proj_X (y)$ be the set of points $x$ in $X$ such that $d(y, x)
\le d(y, X) + \delta$. Define \textit{the projection} of a subset
$A$ to $X$ as $\proj_X(A) = \cup_{a \in A} \proj_X(a)$.

Let $p$ be a path in $Y$ with initial and terminal endpoints  $p_-$
and $p_+$ respectively. Denote by $\len (p)$ the length of $p$.
Given two points $x, y \in p$, denote by $[x,y]_p$ the subpath of
$p$ going from $x$ to $y$.

Let $p, q$ be two paths in $Y$. Denote by $p\cdot q$(or $p q$ if it
is clear in context) the concatenated path provided that $p_+ =
q_-$.

A path $p$ going from $p_-$ to $p_+$ induces a first-last order as
we describe now. Given a property (P), a point $z$ on $p$ is called
the \textit{first point} satisfying (P) if $z$ is among the points
$w$ on $p$ with the property (P) such that $\len([p_-, w]_p)$ is
minimal. The \textit{last point} satisfying (P) is defined in a
simiarly way.

Let $f(x,y): \mathbb R \times \mathbb R \to \mathbb R_+$ be a
function. For notational simplicity, we frequently write $f_{x,y} =
f(x,y)$.

\subsection{Contracting subsets}

\begin{defn}[Contracting subset]
Suppose $\mathcal L$ is a preferred collection of quasigeodesics in
$X$. Let $\mu: \mathbb R \times \mathbb R \to \mathbb R_+$ and
$\epsilon: \mathbb R \times \mathbb R \to \mathbb R_+$ be two
functions.

Given a subset $X$ in $Y$, if the following inequality holds
$$\diam{\proj_{X} (q)} < \epsilon(\lambda, c),$$
for any $(\lambda, c)$-quasigeodesic $q \in \mathcal L$ with $d(q,
X) \ge \mu(\lambda, c)$, then $X$ is called $(\mu,
\epsilon)$-\textit{contracting} with respect to $\mathcal L$. A
collection of $(\mu, \epsilon)$-contracting subsets is referred to
as a $(\mu, \epsilon)$-\textit{contracting system} (with respect to
$\mathcal L$).

\end{defn}

\begin{example} We note the following examples in various contexts.
\begin{enumerate}
\item
Quasigeodesics and quasiconvex subsets are contracting with respect
to the set of all quasigeodesics in hyperbolic spaces. These are
best-known examples in the literature.
\item
Fully quasiconvex subgroups (and in particular, maximal parabolic
subgroups) are contracting with respect to the set of all
quasigeodesics in the Cayley graph of relatively hyperbolic groups
(see Proposition 8.2.4 in \cite{GePo4}). This is the main situation
that we will deal with in Section 3.
\item
The subgroup generated by a hyperbolic element is contracting  with
respect to the set of all quasigeodesics in groups with non-trivial
Floyd boundary (see Proposition 8.2.4 in \cite{GePo4}). Here
hyperbolic elements are defined in the sense of convergence actions
on the Floyd boundary. Note that groups with non-trivial Floyd
boundary include relatively hyperbolic groups \cite{Ge2}, and it is
not yet known whether these two classes of groups coincide.
\item
Contracting segments in CAT(0)-spaces in the sense of in
Bestvina-Fujiwara are contracting here with respect to the set of
geodesics (see Corollary 3.4 in \cite{BF2}).
\item
Any finite neighborhood of a contracting subset is still contracting
with respect to the same $\mathcal L$.
\end{enumerate}
\end{example}

\begin{conv}
In view of examples above, the preferred collection $\mathcal L$ in
the sequel is always assumed to be containing all geodesics in $Y$.
\end{conv}

\begin{defn}[Quasiconvexity]
Let $\sigma: \mathbb R \to \mathbb R_+$ be a function. A subset $X
\subset Y$ is called \textit{$\sigma$-quasiconvex} if given $U \ge
0$,  any geodesic with endpoints in $N_U(X)$ lies in the
neighborhood $N_{\sigma(U)}(X)$.
\end{defn}

Quasiconvexity follows from the above contracting property.

\begin{lem}\label{quasiconvexity}
Let $X$ be a $(\mu, \epsilon)$-contracting subset in $Y$. Then there
exists a function $\sigma: \mathbb R \to \mathbb R_+$ such that $X$
is $\sigma$-quasiconvex.
\end{lem}
\begin{proof}
Given $U \ge 0$, let $\gamma$ be a geodesic with endpoints in
$N_U(X)$. Define $\sigma(U) = 3\max(U, \mu_{1,0}) + \epsilon_{1,0}$.
It suffices to verify that $\gamma \subset N_{\sigma(U)}(X)$.

Let $z$ be a point in $\gamma$ such that $d(z, X) \ge \mu_{1,0}$.
Denote by $p$ the maximal connected segment of $\gamma$ containing
$z$ such that $d(p, X) \ge \mu(1,0)$. Then $\diam{\proj_X(p)} <
\epsilon_{1,0}$. Note that $d(p_-, X) \le \max(U, \mu_{1, 0})$.
Hence, it follows that $$d(z, X) \le d(z, p_-) + d(p_-, X) \le
\sigma(U),$$ which finishes the proof.
\end{proof}

We need a notion of orthogonality of a quasigeodesic path to a
contracting subset.
\begin{defn}[Orthogonality]
Let $X$ be a $(\mu,\epsilon)$-contracting subset in $Y$. Given a
function $\tau: \mathbb R \times \mathbb R \to \mathbb R_+$, a
($\lambda, c$)-quasigeodesic $p$ is said to be
\textit{$\tau$-orthogonal} to $X$ if $\diam{p \cap
N_{\mu(\lambda,c)}(X)} \le \tau(\lambda,c)$.
\end{defn}

The main point of an orthogonal path is that its projection to the
contracting subset is uniformly bounded. In particular, the
following fact will be frequently used later without explicit
mention.
\begin{lem}
Given a $(\epsilon, \mu)$-contracting subset $X$, let $q$ be a
$(\lambda, c)$-quasigeodesic in $\mathcal L$ that is
$\tau$-orthogonal to $X$. Then the following inequality holds
$$\diam {\proj_X(q)} < A_{\lambda,c},$$
where
\begin{equation}\label{A}
A_{\lambda, c} = \mu(\lambda, c) + \tau(\lambda, c)  +
\epsilon(\lambda, c).
\end{equation}
\end{lem}

We now shall introduce an additional property, named bounded
intersection property for a contracting system.
\begin{defn}[Bounded Intersection]
Given a function $\nu: \mathbb R \to \mathbb R_+$, two subsets $X,
X' \subset Y$ have \textit{$\nu$-bounded intersection} if the
following inequality holds
$$\diam{N_U (X) \cap N_U (X')} < \nu(U)$$
for any $U \geq 0$.
\end{defn}
\begin{rem}
Typical examples include sufficiently separated quasiconvex subsets
in hyperbolic spaces, and parabolic cosets in relatively hyperbolic
groups(see Lemma \ref{boundinter}).
\end{rem}

A $(\mu, \epsilon)$-contracting system $\mathbb X$ is said to have
\textit{$\nu$-bounded intersection} if any two distinct $X, X' \in
\mathbb X$ have $\nu$-bounded intersection. A related notion is the
following bounded projection property, which is equivalent to the
bounded intersection property under the contracting assumption as in
Lemma \ref{equivalence} below.

\begin{defn}[Bounded Projection]
Two subsets $X, X' \subset Y$ have \textit{$B$-bounded projection}
for some $B > 0$ if the following holds
$$\diam{\proj_X(X')} < B, \; \diam{\proj_{X'}(X)} < B$$
\end{defn}

\begin{lem}[Bounded intersection $\Leftrightarrow$ Bounded projection]\label{equivalence}
Let $X, X'$ be two $(\mu, \epsilon)$-contracting subsets. Then $X,
X'$ have $\nu$-bounded intersection for some $\nu: \mathbb R \to
\mathbb R_+$ if and only if they have $B$-bounded projection for
some $B > 0$.
\end{lem}
\begin{proof}
$\Rightarrow$: Let $z, w \in \proj_X(X')$ be such that $d(z,w) =
\diam{\proj_X(X')}$. Then there exist $\hat z, \hat w \in X'$ which
project to $z, w$ respectively. Set $B = 2\epsilon_{1, 0} +
\nu(\mu_{1,0})$.

Let's consider $d(z, X)> \mu_{1, 0}$ and $d(w, X)> \mu_{1, 0}$.
Other cases are easier. Let $p$ be a geodesic segment between $z,
w$. Let $\hat u, \hat v$ be the first and last points respectively
on $p$ such that $d(\hat u, X) \le \mu_{1, 0}, d(\hat v, X) \le
\mu_{1, 0}$. Let $u, v$ be a projection point of $\hat u, \hat v$ to
$X$ respectively. Then $d(u, v) \le \nu(\mu_{1,0})$ by the
$\nu$-bounded intersection of $X, X'$. Since $X$ is a $(\mu,
\epsilon)$-contracting subset, we obtain that $d(z, u) \le
\epsilon_{1, 0}, d(w, v) \le \epsilon_{1, 0}$. Hence
$$d(z, w) \le d(z, u) + d(u, v) + d(v, w) \le B.$$

$\Leftarrow$: Given $U >0$, let $z, w \in N_U(X) \cap N_U(X')$. Let
$\hat z, \hat w \in X'$ be such that $d(\hat z, z) \le U, d(\hat w,
z) \le U$. Project $z, w$ to $z', w' \in X$ respectively. Then $d(z,
w) \le d(z', w') + 2U$. It remains to bound $d(z', w')$.

It is easy to verify that the projection of a geodesic segment of
length $U$ on $X$ have a upper bounded size $(2\mu_{1, 0} +
\epsilon_{1, 0} + U)$. Hence $\diam{\proj_X([\hat z, z])} \le
(2\mu_{1, 0} + \epsilon_{1, 0} + U), \diam{\proj_X([\hat w, w])} \le
(2\mu_{1, 0} + \epsilon_{1, 0} + U)$. It follows that
$$
\begin{array}{rl}
d(z' ,w') &\le \diam{\proj_X([\hat z, z])} + \diam{\proj_X(X')} +
\diam{\proj_X([\hat w, w])} \\
&\le B +2(2\mu_{1, 0} + \epsilon_{1, 0} + U).
\end{array}
$$

Then $d(z, w) \le B +4\mu_{1, 0} + 2\epsilon_{1, 0} + 2U$. It
suffices to set $\nu(U) = B +4\mu_{1, 0} + 2\epsilon_{1, 0} + 2U$.
\end{proof}

To conclude this subsection, we note a thin-triangle property when
one side of a triangle lies near a contracting subset. Recall that
the constant $A_{\lambda, c}$ below is defined in (\ref{A}).
\begin{lem}\label{firststep}
Given $X\lambda \ge 1, c \ge 0$, let $\gamma = p q$, where $p$ is a
geodesic and $q$ is $(\lambda,c)$-quasigeodesic in $\mathcal L$.
Assume that $p_-, p_+ \in X \in \mathbb X$ and $q$ is
$\tau$-orthogonal to $X$. Then $\gamma$ is a $(\lambda,
C_{\lambda,c})$-quasigeodesic, where
\begin{equation}\label{C}
C_{\lambda,c} =\lambda(\mu_{\lambda,c}+ \epsilon_{\lambda,c} +
A_{\lambda,c}) + c.
\end{equation}
\end{lem}
\begin{proof} Let $\alpha$ be a geodesic such that $\alpha_- =
\gamma_-, \alpha_+ = \gamma_+$. Let $z$ be the last point on
$\alpha$ such that $d(z, X) \le \mu_{\lambda,c}$. Project $z$ to a
point $z'$ on $X$. Then $d(z, z') \leq \mu_{\lambda,c}$. By
projection, we have
$$\begin{array}{rl}
d(z, q_-) &\le \diam{\proj_{X}([\alpha_-, z]_\alpha)} +
\diam{\proj_{X}(q)} \\
&\le \epsilon_{1,0} + A_{\lambda,c}.
\end{array}
$$
Then we have
$$\begin{array}{rl}
d(z, q_-) &\le d(q_-, z) + d(z, z')\\
&\le \mu_{1,0}+ \epsilon_{1,0} + A_{\lambda,c}.
\end{array}
$$
Hence $\len(\gamma) = \len(p) + \len(q) < \lambda d(\gamma_-,
\gamma_+) + c \le \lambda d(\gamma_-, \gamma_+) + C_{\lambda,c}$.
\end{proof}

\subsection{Admissible Paths}
In this subsection, we give the precise definition of an admissible
path, which is roughly a piecewise quasigeodesic path with
well-controlled local properties.

Recall that $\mathcal L$ is a preferred collection of quasigeodesics
in $X$ such that $\mathcal L$ contains all geodesics. In what
follows, let $\mathbb X$ be a $(\mu, \epsilon)$-contracting system
in $Y$ with respect to $\mathcal L$. Then each $X \in \mathbb X$ is
$\sigma$-quasiconvex, where $\sigma$ is given by Lemma
\ref{quasiconvexity}.

Fix also two functions $\nu: \mathbb R \to \mathbb R_+$ and $\tau:
\mathbb R\times \mathbb R \to \mathbb R_+$, which the reader may
have in mind are the bounded intersection function and orthogonality
function respectively.

\begin{defn}[Admissible Paths]\label{AdmDef}
Given $D \ge 0, \lambda\geq 1, c \geq 0$, a \textit{$(D, \lambda,
c)$-admissible path} $\gamma$ is a concatenation of $(\lambda,
c)$-quasigeodesics in $Y$ such that the following conditions hold:

\begin{enumerate}
\item
Exactly one quasigeodesic $p_i$ of any two consecutive ones in
$\gamma$ has two endpoints in a contracting subset $X_i \in \mathbb
X$,

\item
Each $p_i$ has length bigger then  $\lambda D + c$, except that
$p_i$ is the first or last quasigeodesic in $\gamma$,

\item
For each $X_i$,   the quasigeodesics with one endpoint in $X_i$ are
$\tau$-orthogonal to $X_i$, and

\item
Either any two $X_i, X_{i+1}$(if defined) have $\nu$-bounded
intersection, or the quasigeodesic $q_{i+1}$ between them has length
bigger then $\lambda D + c$.
\end{enumerate}
\end{defn}
\begin{rem}
Note that if $\mathbb X$ has $\nu$-bounded intersection, then the
condition (3) is always satisfied.
\end{rem}

For definiteness in the sequel, usually write $\gamma = p_0 q_1 p_1
\ldots q_n p_n$ and assume that $p_i$ has endpoints in a contracting
subset $X_i \in \mathbb X$ and the following conditions hold.

\begin{enumerate}
\item
$\ell(p_{i}) > \lambda D + c$ for $0 < i < n$.
\item
$q_i \in \mathcal L$ is $\tau$-orthogonal to both $X_{i-1}$ and
$X_i$ for $1\leq i \leq n$.
\item
For $1\leq i \leq n$, either $\len(q_i)
> \lambda D + c$, or $X_{i-1}$ and $X_i$ have $\nu$-bounded
intersection.
\end{enumerate}

\begin{rem}
The collection $\{X_i\}$ therein will be referred to as the
(associated) contracting subsets for $\gamma$. It is not required
that $X_i \ne X_j$ for $i \ne j$. This often facilitates the
verification of a path being admissible.
\end{rem}

\begin{defn}[Fellow Traveller]\label{Fellow}
Assume that $\gamma = p_0 q_1 p_1 ... q_n p_n$ is a $(D, \lambda,
c)$-admissible path, where each $p_i$ has two endpoints in $X_i \in
\mathbb X$. Let $\alpha$ be a path such that $\alpha_- = \gamma_-,
\alpha_+=\gamma_+$.

Given $R >0$, the path $\alpha$ is a \textit{$R$-fellow traveller}
for $\gamma$ if there exists a sequence of successive points $z_i,
w_i$($0 \le i \le n$) on $\alpha$ such that $d(z_i, w_i) \ge 1$ and
$d(z_i, (p_{i})_-) < R, \;d(w_i, (p_{i})_+) < R.$
\end{defn}

\subsection{Quasi-geodesicity of long admissible paths}
The aim of this subsection is to show that for a sufficiently large
$D$, a $(D, \lambda, c)$-admissible path is a quasigeodesic. The
main technical result of this subsection can be stated as follows.
\begin{prop}\label{admissible}
Given $\lambda \geq 1, c \geq 0$, there are constants $D=D(\lambda,
c) >0, R = R(\lambda, c)>0$ such that the following statement holds.

Let $\gamma$ be a $(D_0, \lambda, c)$-admissible path for $D_0 > D$.
Then any geodesic $\alpha$ between $\gamma_-$ and $\gamma_+$ is a
$R$-fellow traveller for $\gamma$.
\end{prop}

The main corollary is that a long admissible path is a
quasigeodesic.

\begin{cor}\label{maincor}
Given $\lambda \geq 1, c \geq 0$, there are constants $D=D(\lambda,
c) >0, \Lambda = \Lambda(\lambda, c)\ge 1$ such that given any $D_0
>D$ the $(D_0, \lambda, c)$-admissible path is a $(\Lambda,
0)$-quasigeodesic.
\end{cor}
\begin{proof}
Let $D = D(\lambda, c),\; R = R(\lambda, c)$ be given by Proposition
\ref{admissible}. Then it suffices to set $\Lambda = \lambda(6R + 1)
+ 3c$ to complete the proof.
\end{proof}

The reminder of this subsection is devoted to the proof of
Proposition \ref{admissible}.

We now define, a priori, the candidate constants which are
calculated in the course of proof:
$$
R = R(\lambda, c) =  \max \{(\ref{R1}), \;(\ref{R2}), \;(\ref{R3})),
$$
and
$$
D = D(\lambda,c) =  \max\{(\ref{D1}),\; (\ref{D2})\;, (\ref{D3})\;,
(\ref{D4}), \;(\ref{D5})\}
$$

Let $\gamma = p_0 q_1 p_1 \ldots q_n p_n$ be a $(D_0, \lambda,
c)$-admissible path for $D_0 > D$, where $p_i, q_i$ are $(\lambda,
c)$-quasigeodesics. For definiteness, assume that each $p_i$ has
endpoints in $X_i \in \mathbb X$. Moreover, we can assume that each
$p_i$ is a geodesic, as the general case follows as a direct
consequence.

The proof of Proposition \ref{admissible}  is achieved by the
induction on the number of contracting subsets $\{X_i \}$ for
$\gamma$.

We start with a lemma describing the subpath of an admissible path
around a contracting subset. Denote by $\proj_k (q)$ the projection
of $q$ to $X_k$.

\begin{lem}[Near contracting subsets]\label{neartarget}
Let $X_k$$(0 \le k \le n)$ be a contracting subset for $\gamma$.
Then we have the following
$$\forall k >0: \diam {\proj_k(p_{k-1}  q_{k})} < B_{\lambda,c}, $$
and
$$\forall k < n: \diam {\proj_k(q_{k+1}  p_{k+1})} < B_{\lambda,c},$$
where
$$
B_{\lambda,c} = 2\epsilon_{1,0} + 2\mu_{1,0} + \nu(\mu_{1,0} +
\sigma_0) + A_{\lambda,c}.
$$
\end{lem}

\begin{proof}
We only prove the inequality for the case $p_{k-1}q_{k}$. The other
case is similar. We claim the following inequality
\begin{equation}\label{Q1}
\diam {p_{k-1} \cap N_{\mu(1,0)}(X_k)} < \nu(\mu_{1,0} + \sigma_0) ,
\end{equation}
from which the conclusion follows. In fact, assuming the inequality
(\ref{Q1}) is true. Let $z$(resp. $w$) be the first(resp. last)
point of $p_{k-1}$ such that $z, w \in N_{\mu(1,0)}(X_k).$ Then we
have
$$
\begin{array}{rl}
\diam {\proj_k(p_{k-1}  q_k)} &<\diam {\proj_k([(p_{k-1})_-,
z]_{p_{k-1}})} + \diam {\proj_k([z,
w]_{p_{k-1}})} \\
& + \diam {\proj_k([w, (p_{k-1})_+]_{p_{k-1}})} + \diam
{\proj_k(q_k)}\\
&<2\epsilon_{1,0} + (2\mu_{1,0} + \nu(\mu_{1,0} + \sigma_0)) +
A_{\lambda,c} <B_{\lambda,c}.
\end{array}
$$

In order to prove (\ref{Q1}), we examine the following two cases by
the definition of an admissible path.

\textit{Case 1:} $\len (q_k) > \lambda D + c$. We show that $p_{k-1}
\cap N_{\mu(1,0)}(X_k) = \emptyset$ and hence (\ref{Q1}) holds
trivially. Suppose not. Let $w$ be the last point on $p_{k-1}$ such
that $d(w, X_k) \le \mu_{1,0}$. Project $w$ to a point $w' \in X_k$.
Then $d(w, w')< \mu_{1,0}$. Using projection, we obtain
$$\begin{array}{rl}
d(w, (q_k)_+) &< d(w, w') + d(w', (q_k)_+)\\
&< \mu_{1,0} + \diam{\proj_k [w, p_{k-1}]_{p_{k-1}}} + \diam{\proj_k (q_k)}\\
&< \mu_{1,0} + \epsilon_{1,0} + A_{\lambda,c}.
\end{array}$$ Since $p_{k-1}
q_k$ is a $(\lambda, C_{\lambda,c})$-quasigeodesic by Lemma
\ref{firststep}, we have that
$$C_{\lambda,c} + \lambda d(w, (q_k)_+)> \len([w, (p_{k-1})_+]_{p_{k-1}}) + \len(q_k).$$
As it is assumed that
\begin{equation}\label{D1}
D> \mu_{1,0} + \epsilon_{1,0} + A_{\lambda,c}+ C_{\lambda,c},
\end{equation}
this gives a contradiction with $\len(q_k) > \lambda D + c$.

\textit{Case 2:} Otherwise $X_{k-1}, X_k$ have $\nu$-bounded
intersection. Then $p_{k-1}$ lies in $N_{\sigma(0)}(X_{k-1})$. By
the bounded projection property, we have $$\diam {p_{k-1} \cap
N_{\mu(1,0)}(X_k)} < \diam {N_{\sigma(0)} (X_{k-1}) \cap
N_{\mu(1,0)}(X_k)} < \nu(\mu_{1,0} + \sigma_0).$$ This establishes
(\ref{Q1}).
\end{proof}

We are ready to start the base step of induction.
\begin{lem}[Base Step]\label{basestep}
Proposition \ref{admissible} is true for $n=1$ and $n=2$.
\end{lem}
\begin{proof}
We shall prove a slightly stronger result: let $\alpha$ be a
geodesic such that $d(\alpha_-, \gamma_-) \le \mu_{1,0},\;
d(\alpha_+, \gamma_+) \le \mu_{1,0}$, then $\alpha$ is a $R$-fellow
traveller for $\gamma$.

\textbf{The case "$n=1$"}. Assume that $\gamma= q_1 p_1 q_2$, where
the geodesic $p_1$ has two endpoints in a contracting subset $X_1$.

Note that $[\alpha_-, \gamma_-]$ and $[\alpha_+, \gamma_+]$ are of
length at most $\mu_{1, 0}$. By projection we have
$$
\diam{\proj_1([\alpha_-, \gamma_-])} \le \epsilon_{1,0} + 3 \mu_{1,
0}, \; \diam{\proj_1([\alpha_+, \gamma_+])} \le \epsilon_{1,0} + 3
\mu_{1, 0}.
$$

We claim that $N_{\mu(1, 0)}(X_1) \cap \alpha \neq \emptyset$.
Suppose not. Then we can estimate by projection
$$
\begin{array}{rl}
\len (p_1) & \le \diam {\proj_1(q_1)} + \diam {\proj_1(\alpha)}
+\diam {\proj_1(q_2)} + \diam{\proj_1([\alpha_-, \gamma_-])} +
\diam{\proj_1([\alpha_+, \gamma_+])} \\
& \le 2A_{\lambda, c} + 3\epsilon_{1, 0} + 6 \mu_{1, 0}.
\end{array}
$$
This gives a contradiction as it is assumed that
\begin{equation}\label{D2}
D > 2A_{\lambda, c} + 3\epsilon_{1, 0} + 6 \mu_{1, 0}.
\end{equation}

Let $z$ and $w$ be the first and last points of $\alpha$ such that
$z, w \in N_{\mu(1, 0)}(X_1).$ Project $z, w$ to $z', w'$ to $X_1$
respectively. Hence we see
$$
\begin{array}{rl}
d((q_1)_+, z) &\le \diam {\proj_1(q_1)} + \diam {\proj_1([\alpha_-,
z]_\alpha)} + \diam{\proj_1([\alpha_-, \gamma_-])} + d(z, z') \\
& \le A_{\lambda, c} + 2\epsilon_{1, 0} + 4\mu_{1, 0} < R -1,
\end{array}
$$ as it is assumed that
\begin{equation}\label{R1}
R > A_{\lambda, c} + 2\epsilon_{1, 0} + 4\mu_{1, 0} + 1.
\end{equation}
It is similar that $d((q_2)_-, w) < R -1$. Up to a slight
modification of $z, w$, we see that $\alpha$ is a $R$-fellow
traveller for $\gamma$.

\textbf{The case "$n=2$"}. This case is similar to the case "n=1".
We only indicate the necessary changes in the below.

Assume that $\gamma = q_1 p_1 q_2 p_2 q_3$, where the geodesic $p_1,
p_2$ have two endpoints in contracting subsets $X_1, X_2$
respectively.

We first claim that $N_{\mu(\lambda, c)}(X_1) \cap q_3 = \emptyset$.
If not, let $z$ be the first point on $q_3$ such that $z \in
N_{\mu(\lambda, c)}(X_1) \cap q_3$. Project $z$ to $z'$ on $X_1$. By
Lemma \ref{neartarget}, we see that
$$
d(z', (q_2)_-) \le \diam {\proj_1(q_2 p_2)} + \diam
{\proj_1([(q_3)_-, z]_{q_3})} \le B_{\lambda, c} +
\epsilon_{\lambda, c}.
$$
The case "$n=1$" shows that $q_2 p_2 q_3$ is a $(\Lambda,
0)$-quasigeodesic, where $\Lambda = \Lambda(\lambda, c)$ is given by
Corollary \ref{maincor}. It follows that
$$
\len(p_2) \le \len([(q_2)_-, z]_{\gamma}) \le \Lambda d((q_2)_-, z)
\le \Lambda(B_{\lambda, c} + \epsilon_{\lambda, c} + \mu_{\lambda,
c}).
$$
This gives a contradiction as it is assumed that
\begin{equation}\label{D3}
D > \Lambda(B_{\lambda, c} + \epsilon_{\lambda, c} + \mu_{\lambda,
c}).
\end{equation}
Hence it is shown that $N_{\mu(\lambda, c)}(X_1) \cap q_3 =
\emptyset$.

Using the same argument as the case "$n=1$", we can see that $\alpha
\cap N_{\mu(1, 0)} (X_i) \neq \emptyset$ for $i= 1, 2$. Let $z_1$
and $w_1$ be the first and last points of $\alpha$ such that $z_1,
w_1 \in N_{\mu(1, 0)}(X_1).$ Then as in the case "$n=1$", we obtain
$$d((p_1)_-, z_1) \le A_{\lambda, c} + 2\epsilon_{1, 0} + 4\mu_{1, 0} < R -1,$$ and
$$
\begin{array}{rl}
d((p_1)_+, w_1) & \le \diam {\proj_1(q_2 p_2)} + \diam
{\proj_1(q_3)}
+ \diam {\proj_1([w_1, \alpha_+]_\alpha)} + \diam{\proj_1([\alpha_+, \gamma_+])}\\
& \le B_{\lambda, c} + 3\epsilon_{1, 0} + 4\mu_{1, 0} < R -1,
\end{array}
$$ as it is assumed that
\begin{equation}\label{R2}
R > B_{\lambda, c} + 3\epsilon_{1, 0} + 4\mu_{1, 0} + 1.
\end{equation}

Consider the path $\gamma' = [w_1', (p_1)_+] q_2 p_2 q_3$ which is
$(D, \lambda, c)$-admissible. Let $\alpha' = [w_1,
\alpha_+]_\alpha$. Let $z_2$ and $w_2$ be the first and last points
of $\alpha'$ such that $z_2, w_2 \in N_{\mu(1, 0)}(X_2).$ Similarly
as above, we obtain that $d((p_2)_-, z_2) \le R-1, d((p_2)_+, w_2)
\le R-1$.

Consequently, it is shown that $\alpha$ is a $R$-fellow traveller
for $\gamma$.
\end{proof}

\textbf{Inductive Assumption}: Assume that Proposition
\ref{admissible} holds for any $(D_0, \lambda,c)$-admissible path
$\gamma'$ with $k$ contracting subsets $X_i \in \mathbb X$ ($k \le
n$). Then any geodesic between $\gamma'_-, \gamma'_+$ is a
$R$-fellow traveller for $\gamma'$. Moreover $\gamma'$ is a
$(\Lambda,0)$-quasigeodesic, where $\Lambda = \Lambda(\lambda, c)$
is given by Corollary \ref{maincor}.

We now consider the admissible path $\gamma = p_0 q_1 p_1 \ldots q_n
p_n$, which has $n+1$ contracting subsets $\{X_i \in \mathbb X: 0
\le i \le n\}$.
\begin{lem}[Far from contracting subsets]\label{fartarget}
Let $X_k$ $(0 <k < n)$ be a contracting subset for $\gamma$. Denote
$\beta = [(p_0)_-,(q_{k-2})_+]_\gamma$. Then the following holds
$$\beta \cap N_{R+ \sigma(\mu(1,0))}(X_k) = \emptyset.$$
\end{lem}
\begin{proof}
Suppose, to the contrary, that $\beta \cap N_{R+
\sigma(\mu(1,0))}(X_k) \ne \emptyset$. Let $z$ be the last point on
$\beta$ such that $d(z, X_k) \leq R + \sigma_{\mu(1,0)}$. Project
$z$ to $w$ on $X_k$.

Observe that $\hat \beta = \beta \cup p_{k-1} q_k [(p_k)_-, w]$ is a
$(D_0, \lambda,c)$-admissible path with at most $n$ contracting
subsets, as $k < n$. Hence $\bar \beta$ is a $(\Lambda,
0)$-quasigeodesic by Inductive Assumption. It follows that
$$\len([z, w]_{\hat \beta}) < \Lambda(R + \sigma_{\mu(1,0)}).$$ This
gives a contradiction with $\len (p_{k-1})
> D$, as it is assumed that
\begin{equation}\label{D4}
D > \Lambda(R  + \sigma_{\mu(1,0)}).
\end{equation}

Therefore, the segment $\beta$ has at least a distance $R +
\sigma_{\mu(1,0)}$ to $X_k$.
\end{proof}

\begin{figure}[htb] 
\centering \scalebox{0.7}{
\includegraphics{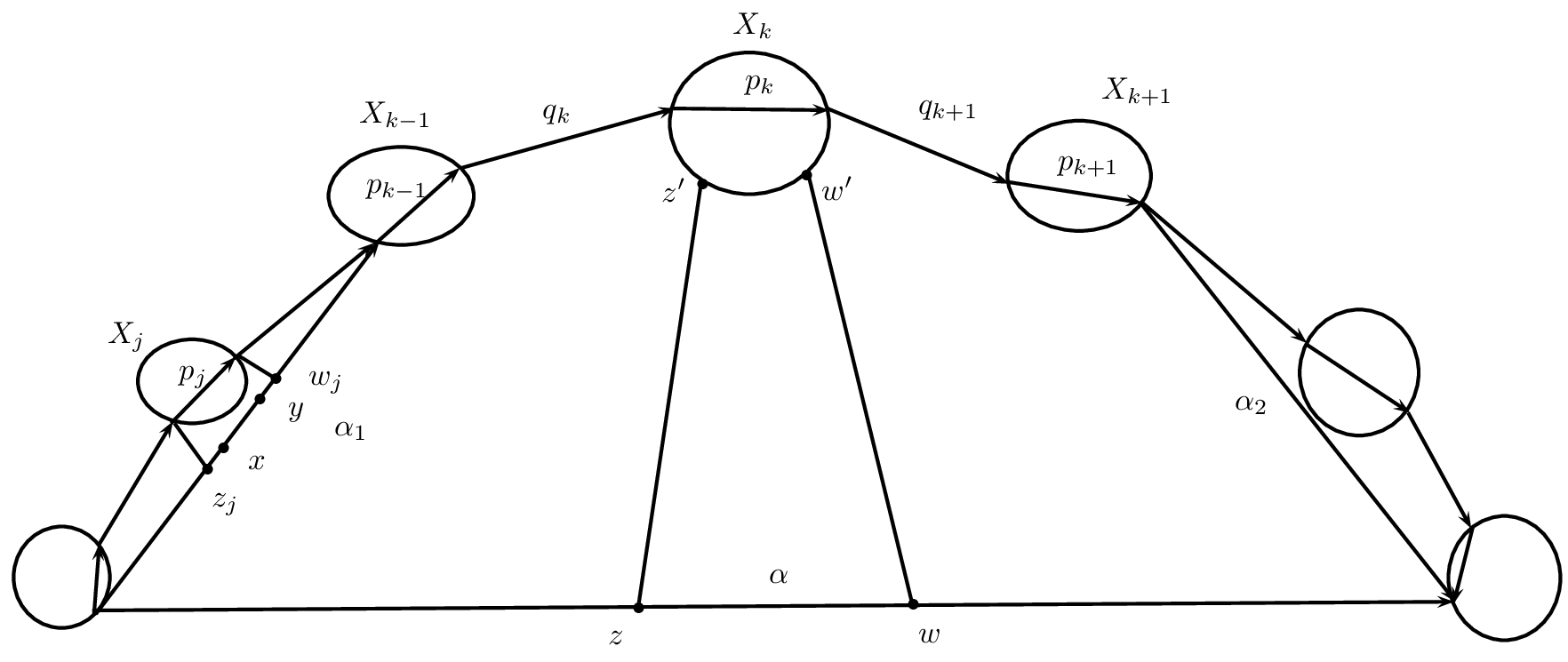} 
} \caption{Proof of Proposition \ref{admissible}} \label{fig:fig1}
\end{figure}

A key step in proving Proposition \ref{admissible} is the following.
\begin{lem}\label{Rnear}
Let $X_k$ be a contracting subset for $\gamma$, where $0 < k < n$.
Assume $\alpha$ is a geodesic such that $d(\gamma_-, \alpha_-) \le
\mu_{1, 0},\; d(\gamma_+, \alpha_+) \le \mu_{1, 0}$. Then  there
exist points $z, w \in \alpha \cap N_{\mu(1, 0)}(X_k)$ such that
$d(z, (p_k)_-) < R-1, \; d(w, (p_k)_+) < R-1$.
\end{lem}

\begin{proof}
Let $\alpha_1 = [\gamma_-, (p_{k-1})_-]$, $\alpha_2 = [(p_{k+1})_+,
\gamma_+]$, $\beta_1 = [\gamma_-, (p_{k-1})_-]_\gamma$ and $\beta_2
= [(p_{k+1})_+, \gamma_+]_\gamma$. Note that $\alpha_1, \alpha_2,
\beta_1, \beta_2$ may be trivial.

Apply Induction Assumption to $\beta_1$. It follows that $\alpha_1$
is a $R$-fellow traveller for $\beta_1$. Let $z_j, w_j \in \alpha_1$
be the points given by Definition \ref{Fellow}. Note that  $d(z_j,
w_j) \ge 1$ and $d((p_j)_-, z_j) < R$, $d((p_j)_+, w_j) < R$.

Let $x, y$ be the first and last points on $\alpha_1$ respectively
such that $x, y \in N_{\mu(1,0)}(X_k)$. See Figure \ref{fig:fig1}.

\begin{claim}
The segment $[x, y]_{\alpha_1}$ contains no points from $\{z_j, w_j:
1 < j < k-1\}$. Moreover, any geodesic segment  $[(p_j)_-, z_j]$ and
$[(p_j)_+, w_j]$ have at least a distance $\mu_{1,0}$ to $X_k$.
\end{claim}
\begin{proof}[Proof of Claim]
Suppose, to the contrary, that there is a point, say $z_j$ from
$\{z_j, w_j: 1 < j < k-1\}$ such that $z_j \in [x, y]_{\alpha_1}$.
Note that any point on $[x, y]_{\alpha_1}$ has at most a distance
$\sigma_{\mu(1,0)}$ to $X_k$. As $d(z_j, (p_j)_-) < R$, we have
$d((p_j)_-, X_k) < R + \sigma_{\mu(1,0)}$. This contradicts Lemma
\ref{fartarget}, as it follows that $\beta_1$ has at least a
distance $(R + \sigma_{\mu(1,0)})$ to $X_k$.

By the same argument, one can see that any geodesic segment
$[(p_j)_-, z_j]$ and $[(p_j)_+, w_j]$ have at least a distance
$\mu_{1,0}$ to $X_k$.
\end{proof}
By the \textbf{Claim} above, we assume that $x, y \in [z_j,
w_j]_\alpha$ for some $z_j, w_j \in \{z_j, w_j: 1 < j < k-1\}$. (The
case that $x, y \in [w_j, z_{j+1}]_\alpha$ is similar).

Using projection, we shall show that $\alpha \cap N_{\mu(1,0)} (X_k)
\neq \emptyset$. Suppose not. The length of $p_k$ is estimated as
follows:
\begin{equation}\label{L}
\begin{array}{rl}
\diam{\proj_k([\gamma_-, (p_k)_-]_\gamma)} & <  \diam{\proj_k([\alpha_-, z_j]_{\alpha_1})} + \diam{\proj_k ([(p_j)_-, z_j])} \\
& + \diam{\proj_k (p_j)} + \diam{\proj_k [(p_j)_+, w_j]}  \\
& + \diam{\proj_k ([w_j, (\alpha_1)_+]_{\alpha_1})} + \diam {\proj_k (p_{k-1} q_k)}  \\
&< 5\epsilon_{1,0} + B_{\lambda,c}.
\end{array}
\end{equation}

Similarly, we obtain that
\begin{equation}\label{R}
\diam{\proj([(p_k)_+, \gamma_+]_\gamma)} < 5\epsilon_{1,0} +
B_{\lambda,c}.
\end{equation}

Note that $[\alpha_-, \gamma_-]$ and $[\alpha_+, \gamma_+]$ are of
length at most $\mu_{1, 0}$. Hence we have
\begin{equation}\label{M}
\diam{\proj_k([\alpha_-, \gamma_-])} \le \epsilon + 3 \mu_{1, 0}, \;
\diam{\proj_k([\alpha_+, \gamma_+])} \le \epsilon + 3 \mu_{1, 0}.
\end{equation}

It follows from (\ref{L}), (\ref{M}) and (\ref{R}) that
$$\begin{array}{rl} \len(p_k) &< \diam{\proj_k([\gamma_-,
(p_k)_-]_\gamma)} + \diam{\proj_k([(p_k)_+, \gamma_+]_\gamma)} +
\diam{\proj_k(\alpha)} \\
& + \diam{\proj_k([\alpha_-, \gamma_-])} + \diam{\proj_k([\alpha_+, \gamma_+])} \\
& < 13\epsilon_{1,0} + 6\mu_{1, 0} + 2B_{\lambda,c}.
\end{array}$$
This gives a contradiction with $\len(p_k) > D$, as it is assumed
that
\begin{equation}\label{D5}
D > 13\epsilon_{1,0} + 6\mu_{1, 0} + 2B_{\lambda,c}.
\end{equation}

Hence $\alpha \cap N_{\mu(1,0)} (X_k) \neq \emptyset$. Let $z$ be
the first point of $\alpha$ such that $d(z, X_k) \leq \mu_{1,0}$.
Let $z'$ be a projection point of $z$ to $X_k$. Then it follows that
$$
\begin{array}{rl}
d(z, (p_k)_-) &< d(z,z') + d(z', (p_k)_-) \\
&< \mu_{1,0} + 5\epsilon_{1,0} + B_{\lambda,c} < R -1,
\end{array}$$
as it is assumed that
\begin{equation}\label{R3}
R > \mu_{1,0} + 5\epsilon_{1,0} + B_{\lambda,c} +1.
\end{equation}

Let $w$ be the last point of $\alpha$ such that $d(z, X_k) \leq
\mu_{1,0}$. Arguing in the same way, we see that $d(w, (p_k)_+) < R
-1$. This completes the proof.
\end{proof}

We now finish the proof of Proposition \ref{admissible}, which is
repeated applications of Lemma \ref{Rnear}.

\begin{proof}[Proof of Proposition \ref{admissible}]
Recall that $\gamma$ is a $(D, \lambda, c)$-admissible path with at
most $(n+1)$ contracting subsets. Let $\alpha$ be a geodesic such
that $\alpha_- = \gamma_-$ and $\alpha_+ = \gamma_+$. By Lemma
\ref{basestep}, we can assume that $n \ge 2$.

Consider a contracting subset $X_k$ for $\gamma$, where $0 < k < n$.
By Lemma \ref{Rnear}, there exist $z_k, w_k \in \alpha \cap
N_{\mu(1, 0)}$ such that $d(z_k, (p_k)_-) < R-1,\; d(w_k, (p_k)_+) <
R-1$.

Let $w_k'$ be a projection point of $w_k$ to $X_k$. Then $d(w_k,
w_k') \le \mu_{1, 0}$. We consider $j = k + 1$. Let $\gamma' =
[w_k', (p_k)_+][(p_k)_+, \gamma_+]_\gamma$ and $\alpha' = [w_k,
\alpha_+]_\alpha$.

Observe that $\gamma'$ is a $(D, \lambda, c)$-admissible path with
at most $n$ contracting subsets. Apply Lemma \ref{Rnear} to
$\gamma'$ and $\alpha'$. Then there exist points $z_j, w_j \in
\alpha' \cap N_{\mu(1,0)}(X_j)$ such that $d(z_j, (p_j)_-) <R-1,
d(w_j, (p_j)_+) < R-1$.

Continuously increasing or decreasing $k$, the points $z_j, w_j$ on
$\alpha$ are obtained to satisfy $d(z_j, (p_j)_-) <R-1,\; d(w_j,
(p_j)_+) < R-1$ for all $0 \le j \le n$.

The conclusion that $\alpha$ is a $R$-fellow traveller follows from
a slight modification $z_j, w_j$ such that $d(z_j, w_j)\ge 1,\;
d(z_j, (p_j)_-) <R, \; d(w_j, (p_j)_+) < R$. The proof is now
complete.
\end{proof}

\section{Fully quasiconvex subgroups}

In this section,  a finitely generated group $G$ is always assumed
to be hyperbolic relative to $\mathbb P = \{P_i: 1 \le i \le n\}$.
We refer the reader to \cite{Gro}, \cite{Farb}, \cite{Bow1},
\cite{DruSapir} and \cite{Hru} for the references on the relative
hyperbolicity of a group.

Given a finite generating set $S$, let $\Gx$ be the Cayley graph of
$G$ with respect to $S$. We denote by $\lab(\cdot)$ the label
function assigning for a combinatorial path in $\Gx$ the product of
generators labeling its edges in $G$.

Let $Y = \Gx$, $d$ be the combinatorial metric induced on the graph
$\Gx$ and $\mathbb X = \{gP: g \in G, P\in \mathbb P\}$. The
conjugate of a subgroup of $P \in \mathbb P$ is called a
\textit{parabolic subgroup}, a left $P$-coset a \textit{parabolic
coset}.

\subsection{Relatively quasiconvex subgroups} Relative
quasiconvexity of a subgroup has been extensively studied from
different points of view. See \cite{Hru} and \cite{GePo4} for the
equivalence of various definitions.

In this subsection, we shall recall a definition of relatively
quasiconvex subgroup in terms of the geometry of Cayley graphs. This
definition rather replies on the fact that $\mathbb X$ is a
contracting system with bounded intersection in $Y$.

\begin{lem}[Bounded intersection] \label{boundinter} \cite{DruSapir}
There exists a positive real-valued function $\nu: \mathbb R \to
\mathbb R_+$ such that any two distinct $gP, g'P' \in \mathbb X$
have $\nu$-bounded intersection.
\end{lem}

The notion of transition points was introduced by Hruska in
\cite{Hru} and further generalized by Gerasimov-Potyagailo in
\cite{GePo4}.
\begin{defn}
Let $p$ be a path in $\Gx$ and $v$ a point in $p$. Given $U >0,
L>0$, we say $v$ is \textit{$(U, L)$-deep} in some $X \in \mathbb X$
if $\diam{[v, p_-]_p \cap N_U(X)}>L$ and $\diam{[v, p_+]_p \cap
N_U(X)}
> L.$ If $v$ is not $(U, L)$-deep in any $X \in \mathbb X$, then $v$
is called a \textit{$(U, L)$-transition point} of $p$.
\end{defn}

\begin{rem}\label{unqiuedeep}
Let $v$ be a $(U, L)$-deep point. If $L > \nu(U)$, then $v$ is $(U,
L)$-deep in a unique $X \in \mathbb X$ by Lemma \ref{boundinter}.
\end{rem}

The following lemma is clear by the $\nu$-bounded intersection and
Remark \ref{unqiuedeep}. It roughly says that the points at which a
path exits a parabolic coset is transitional.
\begin{lem}\label{boundpt}
Given $X \in \mathbb X, U > 0$, let $z, w$ be the first and last
points on a path $p$ in $\Gx$ such that $d(z, X) \le U, d(w, X) \le
U$. If $d(z, w) > \nu(U)$, then $z, w$ are $(U, \nu(U))$-transition
points of $p$.
\end{lem}

In terms of transition points, we can state a week Morse Lemma for a
pair of quasigeodesics in the Cayley graph $\Gx$. It was originally
proved in \cite{Hru} that the Hausdorff distance between the
transition points of an (absolute) geodesic and vertices of a
relative geodesic is bounded. The following general version
essentially follows from \cite[Proposition 5.2.3]{GePo4}.

\begin{lem}\label{transpoints}
Given $\lambda \ge 1, c\ge 0$, there exists a constant $U =
U(\lambda, c)>0$ such that the following holds for any two
$(\lambda, c)$-quasigeodesics $p, q$ in $\Gx$ with same endpoints.

For any $U_0 \ge U, L>\nu(U_0)$, there is a constant $R=R(U_0, L) >
0$ such that any $(U_0, L)$-transition point of $p$ has at most a
distance $R$ to a $(U_0, L)$-transition point of $q$.
\end{lem}
\begin{rem}
Fix $U_0 \ge U(\lambda,c)$ and $L_1, L_2 > \nu(U_0)$. By Lemma
\ref{boundpt}, it is easy to see that the set of $(U_0,
L_1)$-transition points of a $(\lambda,c)$-quasigeodesic has a
bounded Hausdorff distance $H$ to the set of its $(U_0,
L_2)$-transition points, where $H$ depends on $L_1, L_2$ only.
\end{rem}

We are now ready to state the definition of relatively quasiconvex
subgroups.

\begin{defn} [Relative quasiconvexity]\label{quasiconvexdef}
Let $U = U(\lambda, c), L = \nu(U) + 1$, where $U(\cdot, \cdot),
\nu(\cdot)$ are given by Lemmas \ref{boundinter} and
\ref{transpoints} respectively. A subgroup $H$ of $G$ is called
\textit{relatively $M$-quasiconvex} for some $M >0$ if for any
$(\lambda,c)$-quasigeodesic $p$ in $\Gx$ with endpoints in $H$, any
$(U,L)$-transition point of $p$ lies in $M$-neighborhood of $H$.
\end{defn}

To close this subsection, we recall two results frequently used in
next sections. The first result is a general fact about the
intersection of two subgroups in a countable group.

\begin{lem}\label{intersection}\cite{Hru} \cite{MarPed2}
Suppose $H, K$ be subgroups of a countable group $G$. Let $d$ be a
left invariant proper metric on $G$. Then for any $H, gK, U
>0$, there exists a constant $\kappa=\kappa(H, gK, U)$ such
that $N_U (H) \cap N_U(gK) \subset N_\kappa (H \cap K^g) $.
\end{lem}

The next result is well-known but we could not locate a reference in
the literature. Hence a proof is given here for completeness.
\begin{lem}[Long parabolic intersection]\label{longparabolic}
Suppose $H$ is relatively quasiconvex in $G$.  Given a constant $U >
0$, there exists a constant $L = L(H, U)$ such that if $\diam
{N_U(gP) \cap H} > L$ for some $g \in G, P \in \mathbb P$, then
$\card{H \cap P^g} = \infty$.
\end{lem}
\begin{proof}
Given $U >0$, let $B = \{g: d(1, g) \le U\}$ and $A = \max
\{\kappa(H, gP, U): g \in B, P \in \mathbb P\}$, where $\kappa$ is
the function given by Lemma \ref{intersection}. Consider the finite
collection $\mathbb F = \{H \cap P^g: \card{H \cap P^g} < \infty, g
\in B, P \in\mathbb P\}$. Set $L = \max \{d(1, g): g \in N_A(\cup_{F
\in \mathbb F} F)\}$. We claim that $L$ is the desired constant.

Let $h \in N_U(gP) \cap H$ and $p \in P$ such that $d(h, gp) < U$.
Set $g_0 = h^{-1}gp$. Note that $\diam {N_U(gP) \cap H} = \diam
{N_U(g_0P) \cap H}$, where $d(1, g_0)<U$. Hence by Lemma
\ref{intersection}, $N_U(g_0P) \cap H \subset N_A(P^{g_0} \cap H)$.
This implies that if $\diam {N_U(g_0P) \cap H} > L$, then $P^{g_0}
\cap H$ is infinite. Hence, $P^g \cap H$ is infinite.
\end{proof}

\subsection{Fully quasiconvex subgroups}
The notion of a fully quasiconvex subgroup is the central object in
next sections.
\begin{defn}[Fully quasiconvex subgroups]
Let $H$ be relatively quasiconvex in $G$. Then $H$ is said to be
\textit{fully quasiconvex} if $H \cap P^g$ is either finite or of
finite index in $P^g$ for each $g \in G, P \in \mathbb P$.
\end{defn}

The fundamental fact in this study is that a fully quasiconvex
subgroup is a contracting subset. In particular, a collection of
left cosets of a fully quasiconvex subgroup is a contracting system
which we will focus on in next sections.

\begin{lem}\label{projlem} \cite[Proposition 8.2.4]{GePo4}
Let $H$ be fully quasiconvex in $G$. For any $\lambda \geq 1, c \geq
0$, there exist positive constants $\mu=\mu(\lambda, c)$ and
$\epsilon=\epsilon(\lambda, c)$ such that for any
$(\lambda,c)$-quasigeodesic $\gamma$ satisfying $d(\gamma, H) \ge
\mu(\lambda,c)$ in $\Gx$, we have $\diam{\proj_H \gamma} <
\epsilon(\lambda,c)$.
\end{lem}

\begin{rem}
Note that maximal parabolic subgroups are fully quasiconvex. Hence
by Lemma \ref{boundinter}, $\mathbb X$ is a $(\epsilon,
\mu)$-contracting system with $\nu$-bounded intersection.
\end{rem}

The following lemma is easy exercise by the definition of full
quasiconvexity.
\begin{lem}\label{fullparabolic}
Suppose $H$ is fully quasiconvex in $G$. Then there exists a
constant $U = U(H)$ such that if $\card{H \cap P^g} = \infty$ for
some $P \in \mathbb P, g \in G$ , then $gP \subset N_U(H)$.
\end{lem}

We now come to the key fact that enables us to build a normal form
in the combination theorem in Section 4.
\begin{lem}[Orthogonality]\label{ksideproj}
Suppose $H$ is relatively quasiconvex and $K$ is fully quasiconvex
in $G$. For any constant $U> 0$, there exists a constant
$\tau=\tau(U) > 0$ such that the following statement is true.

Let $h \in H$ be such that $d(1,h) = d(1, C h C)$, where $C = H \cap
K$. Then for any geodesic $q$ between $1$ and $h$, we have $\diam{q
\cap N_U(K)} < \tau(U)$ and $\diam{q \cap N_U(hK)} < \tau(U)$.
\end{lem}
\begin{proof}
By assumption, we have that $d(1,h) = d(1, hC)$ and $d(1,h) = d(1,
Ch)$. It is easy to see that $d(1,h) = d(1, hC)$ implies that
$d(1,h^{-1}) = d(1, Ch^{-1})$. Observe that the verification of the
inequality $\diam{q \cap N_U(hK)} < \tau$ can be reduced to
$\diam{h^{-1}q \cap N_U(K)} < \tau$, where $h^{-1}q$ is a geodesic
between $h^{-1}$ and $1$. Hence, it suffices to verify the following
claim.

\begin{figure}[htb] 
\centering \scalebox{0.7}{
\includegraphics{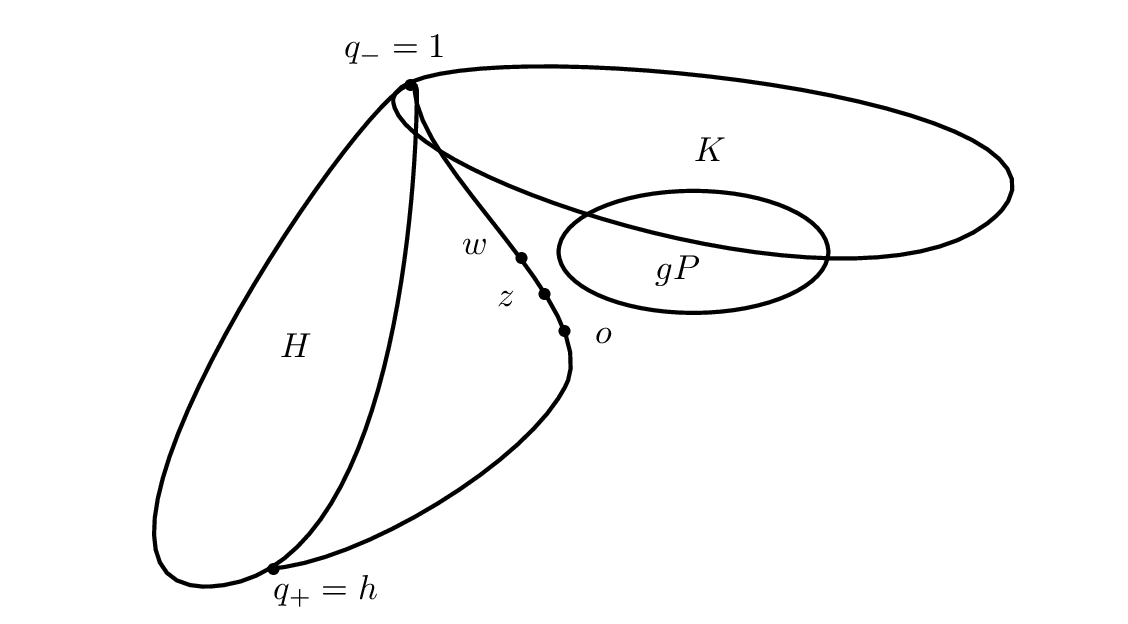} 
} \caption{Proof of Lemma \ref{ksideproj}} \label{fig:fig3}
\end{figure}
\end{proof}

\begin{claim}
Let $d(1, h) = d(1, Ch)$. Then $\diam{q \cap N_U(K)} < \tau$.
\end{claim}
\begin{proof}[Proof of Claim]
Since $H$ is relatively quasiconvex, there exist $U_0, L_0, M$ such
that any $(U_0, L_0)$-transition point of $q$ lies in $N_M(H)$.

Given $U > 0$, let $z$ be the last point on $q$ such that $d(z, K)
\le U$. Since $K$ is fully quasiconvex, we have $d(z, K) \le
\sigma(U)$, where $\sigma$ is given by Lemma \ref{quasiconvexity}.
See Figure \ref{fig:fig3}.

Let $w$ be the last $(U_0, L_0)$-transition point of $q$ such that
$w \in [1, z]_q$. Hence $d(w, H) \le M$. By Lemma
\ref{intersection}, there exists $w' \in H\cap K = C$ such that
$d(w, w') \le \kappa(H, K, M+\sigma(U))$. Since $d(1, h) = d(1,
Ch)$, it follows that $d(1, w) \le \kappa(H, K, M+\sigma(U))$. In
fact, if $d(1, w) > d(w', w)$, then we obtain $$d(1, h) =d(1, w) +
d(w, h)
> d(w', w) + d(w, h) \ge d(1, w'^{-1}h),$$ contradicting the choice
of $h$.

Without loss of generality, we are going to bound $d(1, z)$ under
the assumption that $d(w, z) > \max\{L(H, U), L(K, U)\}$. Note that
$L(H, U), L(K, U)$ are given by Lemma \ref{longparabolic}.

Since $w$ is the last $(U_0, L_0)$-transition point in $[1, z]_q$.
Then $z$ is $(U_0, L_0)$-deep in a unique $gP \in \mathbb X$. In
particular, it follows that $w \in N_{U_0}(gP)$. Indeed, the first
point of $[1, z]_q$ such that $z \in N_{U_0}(gP)$ is a $(U_0,
L_0)$-transition point of $q$ by Lemma \ref{boundpt}.

Since $d(z,w) > L(K, U)$, we see that $\card {K \cap P^g} = \infty.$
Then $gP \subset N_{U_1}(K)$, where $U_1 = U(K)$ is given by Lemma
\ref{fullparabolic}.

Let $o$ be the last point on $[z, q_+]_q$ such that $o \in
N_{U_0}(gP)$. Then $o$ has to be a $(U_0, L_0)$-transition point by
Lemma \ref{boundpt}. By the relative quasiconvexity of $H$, we have
$d(o, H) \le M$. Since $d(o, w) > d(z,w) >L(H, U)$, it follows by
Lemma \ref{longparabolic} that $\card {H \cap P^g} = \infty.$

Note that $o \in N_M(H) \cap N_{U_0}(gP) \subset N_M(H) \cap
N_{U_0+U_1}(K)$. By Lemma \ref{intersection}, there is $o' \in H
\cap K =C$ such that $d(o, o') \le \kappa(H, K, M+U_0+U_1)$.

Since $d(1, h) = d(1, Ch)$, we obtain as above that $d(1, o) \le
d(o, o') \le \kappa(H, K, M+U_0+U_1)$. Hence $d(1, z) \le d(1, o)
\le \kappa(H, K, M+U_0+U_1)$, completing the proof of Claim.
\end{proof}


\section{Combining fully quasiconvex subgroups}

\subsection{The setup.}

Let $H$ be relatively quasiconvex and $K$ fully quasiconvex in a
relatively hyperbolic group $G$. Denote $C = H\cap K$. Let $Y = \Gx$
and $\mathbb X = \{gK: g \in G\}$. We choose $\mathcal L$ to be the
set of all quasigeodesics in $Y$.

Let $\mu, \epsilon$ be the common functions given by Lemma
\ref{projlem} for fully quasiconvex subgroups $K$ and all $P \in
\mathbb P$. Then $\mathbb X$ is a $(\mu, \epsilon)$-contracting
system with respect to $\mathcal L$. ($\mathbb X$ may not have
bounded intersection property.)

Let $\sigma$ be the quasiconvexity function given by Lemma
\ref{quasiconvexity}. Then $H, K$ and each $P \in \mathbb P$ are
$\sigma$-quasiconvex.

\subsection{Normal forms in $ H \star_C K$.}
Let's consider the case $g = k_0 h_1 k_1 ... h_n k_n$, where $h_i
\in H \setminus C, k_i \in K\setminus C$. The other form of $g$ is
completely analogous.

For each $1 \le i \le n$, let $h_i \in H$ be such that $d(1, h_i) =
d(1, C h_i C)$. It is easy to see that such a representation of $g$
always exists.

Let $\gamma = p_0 q_1 p_1 \ldots q_n p_n$ be a concatenation of
geodesic segments $p_i, q_i$ in $\Gx $ such that $(p_0)_- = 1$ and
$\lab (p_i) = k_i, \lab (q_i) = h_i$. We call $\gamma = p_0 q_1 p_1
\ldots q_n p_n$ a \textit{normal path} of $g$.

Define $f_i = k_0 h_1 k_1 \ldots k_{i-1} h_i$ for $1 \le i \le n$
and $f_0 =1$. Then $p_i$ has two endpoints in the left coset $f_i K
\in \mathbb X$.

The following lemma is almost obvious by the definition of a normal
path.

\begin{lem}\label{inject1}
There exists a constant $D=D(H,K) > 0$, $\Lambda=\Lambda(H,K) > 0$
such that the following statement is true.

Suppose that there are $\dot H \subset H$ and $\dot K \subset K$
such that $\dot H \cap \dot K = C$ and $d(1, g) > D$ for any $g \in
(\dot H \cup \dot K) \setminus C$. Then the normal path of any
element in $\dot H \star_C \dot K$ is a $(D, 1, 0)$-admissible path
and thus a ($\Lambda, 0$)-quasigeodesic.
\end{lem}

\begin{proof}
Let $D =D(H, K), \Lambda=\Lambda(H, K)$ be the constants given by
Main Corollary \ref{maincor}.

Let $g = k_0 h_1 k_1 \ldots h_n k_n$ be an element in $\dot H
\star_C \dot K$, where $k_i \in \dot K \setminus C, h_i \in \dot H
\setminus C$. By assumption, it follows that $d(1, h_i) > D, d(1,
k_i)>D$.

Let $\gamma = p_0 q_1 p_1 \ldots q_n p_n$ be the associated normal
path. Hence $\len(p_i) > D, \len(q_i) > D$. By Lemma
\ref{ksideproj}, $q_i$ is $\nu$-orthogonal to $g_{i-1} K, g_i K$.
Hence $\gamma$ is a $(D, 1, 0)$-admissible path. By Main Corollary
\ref{maincor},  $\gamma$ is a ($\Lambda, 0$)-quasigeodesic.
\end{proof}

\begin{rem}
Note that if $K$ is a maximal parabolic subgroup, then it suffices
to assume that $d(1, g) > D$ for any $g \in \dot K \setminus C$.
This follows from the fact that $\{gP: g \in G, P \in \mathbb P \}$
has bounded intersection property. Hence, the second case of
Condition (3) in the definition of admissible paths is always
satisfied.
\end{rem}

\subsection{Proof of Theorem \ref{freeproduct}}
Lemma \ref{inject1} implies that $\dot H \star_C \dot K \to \langle
\dot H, \dot K \rangle$ is injective.

To show the relative quasiconvexity of $\langle \dot H, \dot K
\rangle$, note that the normal path of each element in $\langle \dot
H, \dot K \rangle$ is a $(\Lambda, 0)$-quasigeodesic. Let $U =
U(\Lambda, 0)$ be the constant given by Lemma \ref{transpoints} and
$L = \nu(U) + 1$. Observe that any $(U, L)$-transition point of
$\gamma$ is a $(U, L)$-transition point of either $p_i$ or $q_i$.
Since $H, K$ are relatively quasiconvex,  we see that any $(U,
L)$-transition point of $\gamma$ has a uniform bounded distance to
$\langle \dot H, \dot K \rangle$. Hence by Lemma \ref{transpoints},
$\langle \dot H, \dot K \rangle$ is relatively quasiconvex.

We now show the last statement of Theorem \ref{freeproduct} about
the conjugacy classes of parabolic subgroups.

\begin{lem}\label{amalgparabolic}
Every parabolic subgroup of $\langle \dot{H}, \dot{K}\rangle$ is
conjugated into either $\dot H$ or $\dot K$.
\end{lem}
\begin{proof}
Note that maximal parabolic subgroups in $\langle \dot{H},
\dot{K}\rangle$ are of form $P^f \cap \langle \dot{H},
\dot{K}\rangle$, where $f \in G$ and $P \in \mathbb P$. Let $g \in
\langle \dot{H}, \dot{K}\rangle \setminus (\dot H \cup \dot K)$. The
idea is to take sufficiently large $D$ in Theorem \ref{freeproduct},
to show that $g \notin P^f$ for any $f \in G, P\in \mathbb P$.

Suppose, to the contrary,  that $g = fpf^{-1}$ for some $f \in G, p
\in P$. Without loss of generality, we assume that $g = h_0 k_1 h_1
... k_n h_n$, where $h_i \in \dot H, k_i \in \dot K$. Denote the
normal path of $g$ by $\gamma = p_0 q_1 p_1 ... q_n p_n$.

Let $\alpha$ be a geodesic segment with the same endpoints as
$\gamma$. By Proposition \ref{admissible}, the endpoints of each
$p_i, q_i$ lie in a uniform $R$-neighborhood of $\alpha$, where $R =
R(1,0)$.

Let $z, w$ be the first and last points of $\alpha$ respectively
such that $z, w \in N_{\mu(1, 0)} (fP)$. Then by projection, we have
$d(\alpha_-, z), d(\alpha_+, w) \le d(1, f) + \mu_{1, 0} +
\epsilon_{1, 0}$. Let $\alpha' = [z, w]_\alpha$. Then $\alpha'
\subset N_{\sigma(\mu(1, 0))}(fP)$.

Note that the analysis in the last paragraph applies to any power of
$g$. By taking sufficiently large power of $g$, the length of
$\alpha'$ can be arbitrarily large.

Set $U =\sigma(R+\sigma(\mu_{1, 0}))$. Hence  we can assume further
that there exist consecutive $p_i, q_i$ for some $i$ of $\gamma$
such that $q_i, p_i \subset N_{U}(fP)$. Let $f_i \in G$ be the
element associated to the vertex $(q_i)_+ = (p_i)_-$. Then $q_i,
p_i$ have endpoints in $f_i H, f_i K$ respectively.

Since $\diam{N_U(fP) \cap f_i H} > D$ and we assumed that $D > L(H,
U)$, it follows by Lemma \ref{longparabolic} that $H \cap P^{f'}$ is
infinite, where $f' = f_i^{-1} f$. Similarly, as it is assumed that
$D> L(K, U)$, we have $K \cap P^{f'}$ is infinite. By Lemma
\ref{fullparabolic}, we see that $f' P \subset N_U(K)$.

We now translate the terminal point of $q_i$ to $1$. Note that
$f_i^{-1} q_i \subset N_U(f' P)$. Since $k_i \in K$ is such that
$d(1, k_i) = d(1, C k_i C)$. By Lemma \ref{ksideproj}, we have
$\diam{f_i^{-1} q_i \cap N_U(K)} < \tau(U)$. This implies that
$\len(q_i) < \tau(U)$. It suffices to assume further $D > \tau(U)$
to get a contradiction. Hence, it is shown that any $g \notin P^f$
for $f \in G, P \in \mathbb P$.
\end{proof}

In fact, the proof also shows the following.

\begin{cor}
Any element $g$ in $\langle \dot{H}, \dot{K}\rangle \setminus (\dot
H \cup \dot K)$ is hyperbolic, i.e.: $g$ is not conjugated into any
$P \in \mathbb P$.
\end{cor}

We also note the following corollary.
\begin{cor}
The virtual amalgamation of two fully quasiconvex subgroups is fully
quasiconvex.
\end{cor}

\section{HNN combination theorem}

As in the previous section, a finitely generated group $G$ is
assumed to be hyperbolic relative to a collection of subgroups
$\mathbb P$. In addition, we will also consider the geometry of
\textit{relative Cayley graph} of $G$ with respect to $\mathbb P$,
denoted by $\G$.

Note that $\Gx$ is a subgraph of $\G$, and a path(resp. geodesic) in
$\G$ is also referred to as a \textit{relative} path(resp.
\textit{relative} geodesic).

\subsection{The setup}
Let $Y = \Gx$ and $\mathbb X = \{gP: g \in G, P \in \mathbb P\}$.
Let $\mu, \epsilon$ be the common functions given by Lemma
\ref{projlem} for all $P \in \mathbb P$. We choose $\mathcal L$ to
be the set of all quasigeodesics in $Y$. Then $\mathbb X$ is a
$(\mu, \epsilon)$-contracting system with respect to $\mathcal L$.

Let $\sigma$ be the quasiconvexity function given by Lemma
\ref{quasiconvexity}. Then each $P \in \mathbb P$ are
$\sigma$-quasiconvex.

\subsection{Lift paths}
The notion of a lift path is interacting between the geometry of
relative and normal Cayley graphs. Before giving the definition, we
need recall several notions introduced by Osin \cite{Osin} in
relative Cayley graphs.
\begin{defn}[$P_i$-components]
Let $\gamma$ be a path in $\G$. Given $P_i \in \mathbb P$, a subpath
$p$ of $\gamma$ is called \textit{$P_i$-component} if $\lab(p) \in
P_i$ and no subpath $q$ of $\gamma$ exists such that $p \subsetneq
q$ and $\lab(q) \in P_i$.

Two $P_i$-components $p_1, p_2$ of $\gamma$ are \textit{connected}
if $(p_1)_-, (p_2)_-$ belong to a same $X \in \mathbb X$. A
$P_i$-component $p$ is \textit{isolated} if no other $P_i$-component
of $\gamma$ is connected to $p$.
\end{defn}

\begin{defn}[Lift path]
Let $\gamma$ be a path in $\G$. The \textit{lift path} $\hat \gamma$
is obtained by replacing each $P_i$-component of $\gamma$ by a
geodesic with the same endpoints in $\Gx$.
\end{defn}

We recall a fact implicitly in \cite[Thm. 1.12(4)]{DruSapir} that
the lift path of a relative geodesic is a quasigeodesic. A rather
general version can be found in \cite[Proposition 7.2.2]{GePo4}.

\begin{lem} \label{liftpath}
There exists a constant $\lambda \ge 1$ such that the lift of any
relative geodesic is a $(\lambda, 0)$-quasigeodesic.
\end{lem}

The following result says that a relative geodesic leaves parabolic
cosets in an orthogonal way, as is defined in Section 2.

\begin{lem}[Orthogonality of relative geodesics]\label{relorthogonal}
For any constant $U > 0$, there exists a constant $\tau=\tau(U) >0$
such that the following holds.

Given $gP \in \mathbb X$, let $p$ be a relative geodesic such that
$p \cap gP = \{p_+\}$. Denote by $\hat p$ the lift of $p$. Then
$\diam{\hat p \cap N_U(gP)} < \tau$.
\end{lem}
\begin{proof}
By Lemma \ref{equivalence}, $\mathbb X = \{gP: g \in G, P \in
\mathbb P\}$ has $B$-bounded projection for some $B >0$. Moreover,
the projection of an edge in $\Gx $ to any $X \in \mathbb X$ is also
uniformly bounded by a constant, say $B$ for convenience.

Given $U > 0$, set $\nu(U) = 4B(U+1)^2 + 2(U+1)$. Let $z$ be the
first point of $\hat p$ such that $d(z, gP) \le U$. We shall show
that $d(z, p_+) < \nu(U)$.

Without lost of generality, assume that $z$ is a vertex in a
geodesic segment $\hat s$, which is the lift of a $P_i$-component
$s$ of $p$. Clearly $\dxp(z, s_+) \le 1$. Let $w \in gP$ such that
$\dx(z, w) \le U$.

Let $q$ be a geodesic in $\Gx$ such that $q_- = z, q_+ = w$. Since
$\dxp(w, p_+) \le 1$, let $e_1$ be an edge such that $\lab(e_1) \in
P$ and $(e_1)_-=w, (e_1)_+ =p_+$. Similarly, let $e_2$ be an edge
such that $\lab(e_2) \in P_i$ and $(e_2)_-=s_+, (e_2)_+ =z$.
Consider the cycle $o = q e_1 [p_+, s_+]_p e_2$. Note that
$$\len(o) \le \len(q) + 1 + \dxp(s_+, p_+) + 1 \le 2(U+1).$$

Since $p$ is a relative geodesic, each $P_i$-component of $o$ is
isolated. Then given a $P_i$-component $t$ of $o$, we project other
edges of $o$ to the parabolic coset associated to $t$. This gives
the estimate $d(t_-, t_+) \le B \len (o)$ for each $P_i$-component
$t$ of $o$. Hence, it follows that $d(z, p_+) \le \len([z,
p_+]_{\hat p}) \le \len(o) + B (\len(o))^2 < 4B(U+1)^2 + 2(U+1)$.
\end{proof}

We now consider a class of admissible paths coming from the lifts of
piecewise relative geodesics.  Such type of admissible paths will be
obtained by truncating the normal path defined in the next
subsection.
\begin{lem}\label{reladmissible}
There are constants $D >0, \Lambda \ge 1$ such that the following
holds.

Let $\gamma = p_0 q_1 p_1 \ldots q_n p_n$ be a concatenation path in
$\G$, where $p_i$ are $P_i$-components of $\gamma$ for some $P_i \in
\mathbb P$ and $q_i$ are relative geodesics. Assume that $d((p_i)_-,
(p_i)_+)
> D$ for $0 < i < n$, and $p_{i-1}, p_i$ are not connected. Then the
lift of $\gamma$ is a $(\Lambda, 0)$-quasigeodesic.
\end{lem}
\begin{proof}
Let $\hat \gamma = \hat p_0 \hat q_1 \hat p_1 \ldots \hat q_n \hat
p_n$ be the lift path. Each $\hat q_i$ is a
$(\lambda,0)$-quasigeodesic for some $\lambda \ge 1$ by Lemma
\ref{liftpath}. Let $g_i P_i \in \mathbb X$ be the parabolic coset
in which the endpoints of $p_i$ lie. Note that Lemma
\ref{relorthogonal} verifies that $\hat q_i$ is orthogonal to
$g_{i-1} P_{i-1}, g_i P_i$. Hence, we see that $\hat \gamma$ is a
$(D, \lambda, 0)$-admissible path. Since $\mathbb X$ is a
contracting system. As a consequence, the constants $D, \Lambda$ are
provided by Main Corollary \ref{maincor}.
\end{proof}

\subsection{Normal forms in $H \star_{Q^t = Q'}$}
Let $P \in \mathbb P, f \in G$ be such that $Q = P \cap H$ and $Q^f
= Q'$ are non-conjugate maximal parabolic subgroups of $H$. Denote
$P' = P^f$.

Assume that there is $c \in P$ such that $Q^c = Q$. Set $t = fc$.

Let $g \in H \star_{Q^t = Q'}$ be written as the form $h_1
t^{\epsilon_1} h_2 t^{\epsilon_2} \ldots h_n t^{\epsilon_n}$, where
$h_i \in H, \epsilon_i \in \{1, -1\}$. By Britton's Lemma, if
$\epsilon_i = 1, \epsilon_{i+1}=-1$, then $t \notin Q$; if
$\epsilon_i = -1, \epsilon_{i+1}=1$, then $t \notin Q'$.

A \textit{normal path} of $g$ is a concatenated path $\gamma = q_1
(\beta_1 p_1)^{\epsilon_1} q_2 (\beta_2 p_2)^{\epsilon_2} \ldots q_n
(\beta_n p_n)^{\epsilon_n}$ in $\G $ with the following properties

\begin{enumerate}
\item
$q_i$ is a relative geodesic in $\G$ such that $\lab(q_i) = h_i$,
\item
$\beta_i$ is a geodesic in $\Gx$ such that $\lab(\beta_i) = f$, and
\item
$p_i$ is an edge in $\G$ such that $\lab(p_i) = c$.
\end{enumerate}

Let $g_i P \in \mathbb X$ be the parabolic coset such that $(p_i)_-,
(p_i)_+ \in g_i P$. These $g_i P$ will serve as contracting subsets
for a admissible path that we will construct. We shall first verify
that consecutive $g_i P$ are distinct.
\begin{lem}\label{distneighbor}
Peripheral cosets $g_{i-1} P, g_i P$ are distinct.
\end{lem}
\begin{proof}
In the following, we only verify the case $i=1$. The other cases are
completely analogous.

If $\epsilon_1 = 1, \epsilon_2 = {-1}$, then $\gamma = q_1 (\beta_1
p_1) q_2 (p_2^{-1} \beta_2^{-1})\ldots q_n  (\beta_n
p_n)^{\epsilon_n}$. It follows from Britton's Lemma that $\lab(q_2)
\notin Q$. Hence we see that $g_1 P, g_2 P$ are distinct.

If $\epsilon_1 = 1, \epsilon_2 = 1$, then $\gamma = q_1 (\beta_1
p_1) q_2 (\beta_2 p_2) \ldots q_n  (\beta_n p_n)^{\epsilon_n}$.
Assume that $(p_1)_- = 1$. Let $g_2 P$ be the parabolic coset in
which $p_2$ lies. Suppose, to the contrary, that $P = g_2 P$, that
is $h_2 f \in P$. By assumption, note that $P^f = P'$ and thus
$h_2^{-1} P h_2 = P'$. It follows that $h_2^{-1} Q h_2= Q'$,
contradicting the assumption that $Q, Q'$ are not conjugate in $H$.
\end{proof}

Note that the normal path is defined in $\G $. So our next step is,
before lifting each $p_i, q_i$, to truncate the extra part of
$\gamma$ lying inside $g_i P$ as follows.

\textbf{Truncating the path $\gamma$}. Given $g \in \langle H, t
\rangle$, let $$\gamma = q_1 (\beta_1 p_1)^{\epsilon_1} q_2 (\beta_2
p_2)^{\epsilon_2} \ldots q_n (\beta_n p_n)^{\epsilon_n}$$ be its
normal path. For each $g_i P$, if $q_i \cap g_i P \ne \emptyset$,
then let $z_i$ be the first point of $q_i$ such that $z_i \in g_i
P$; otherwise, let $z_i = (p_i^{\epsilon_i})_-$. In a similar way,
if $q_{i+1} \cap g_i P \ne \emptyset$, then let $w_i$ be the last
point of $q_{i+1}$ such that $w_i \in g_i P$;  otherwise, let $w_i =
(p_i^{\epsilon_i})_+$. Denote $w_0 = \gamma_-$.

Let $q_i'$ be the lift path of the segment $[w_{i-1}, z_i]_\gamma$.
Let $p'_i$ be a geodesic in $\Gx$ between $z_i$ and $w_i$. Then
$\bar \gamma = q_1' p_1' ... q_n' p_n'$ is called the
\textit{truncation} of $\gamma$.

We now carefully examine the truncation paths and show that they are
admissible paths.  Let $\card{f} = d(1, f)$.
\begin{lem}\label{hnnadm}
Given $D > 0$,  assume that $d(1, g) > D$ for any $g \in cQ$. Let
$D' = D-\kappa(H, f^{-1} P, M) - \kappa(H, P, M)$. Then the
truncation path of any element in $\langle H, t \rangle \setminus H$
is a $(D', \lambda, (\lambda+2)\card{f})$-admissible path.
\end{lem}
Recall that the number $\lambda$ above is given by Lemma
\ref{liftpath} and the function $\kappa(\cdot, \cdot, \cdot)$ given
by Lemma \ref{intersection}.

\begin{proof}
Let $\gamma = q_1 (\beta_1 p_1)^{\epsilon_1} q_2 (\beta_2
p_2)^{\epsilon_2} \ldots q_n (\beta_n p_n)^{\epsilon_n}$ be the
normal path of an element in $\langle H, t \rangle$. Without loss of
generality, we consider the case that $\epsilon_1 = 1$. The case
that $\epsilon_1 = -1$ is symmetric by reversing the orientation of
$\gamma$.

Let $(p_1)_- =1$. Let $z$ be the first vertex of $q_1$ such that $z
\in P$ if it exists; otherwise let $z = (\beta_1)_+$. The relative
$M$-quasiconvexity of $H$ implies that $z \in N_M(f^{-1}H) \cap P$.
By Lemma \ref{intersection}, there is $z' \in f^{-1} Hf \cap P =
f^{-1} Q' f = Q$ such that $d(z', z) \le \kappa(H, f^{-1}P, M)$.

Let $w$ be the last vertex of $q_2$ such that $w \in P$. The
relative $M$-quasiconvexity of $H$ implies that $w \in N_M(c H) \cap
P$. By Lemma \ref{intersection}, there is $w' \in H \cap P = Q$ such
that $d(cw', w) \le \kappa(H, P, M)$. See Figure \ref{fig:fig2}.

Let $q_1'$ be the lift of the relative path $[\gamma_-, z]_\gamma$,
and $p_1'$ a geodesic between $z$ and $w$. By Lemma \ref{liftpath},
we see that $q_1'$ is a $(\lambda, (\lambda +
2)\card{f})$-quasigeodesic.

Since $z'^{-1} c w' \in Q c Q = c Q$, it follows that $d(z', cw')
>D$. Then we have

\begin{equation}\label{truncation}
\begin{array}{rl}
\len(p_1') = d(z, w) & > d(z', cw') - d(z, z') -d(w, cw') \\
&> D - \kappa(H, f^{-1}P, M) - \kappa(H, P, M).
\end{array}
\end{equation}
We continuously truncate $q_i$ to define $q_i', p_i'$ as above. Let
$\bar \gamma = q_1' p_1' ... q_n' p_n'$ be the truncation path,
where $p_i'$ are $P$-components. Moreover, the paths $q_i'$ are
$(\lambda, (\lambda + 2)\card{f})$-quasigeodesics. By Lemma
\ref{distneighbor}, we see that $\bar \gamma$ is a $(D', \lambda,
(\lambda + 2)\card{f})$-admissible path.
\end{proof}

\begin{figure}[htb] 
\centering \scalebox{0.8}{
\includegraphics{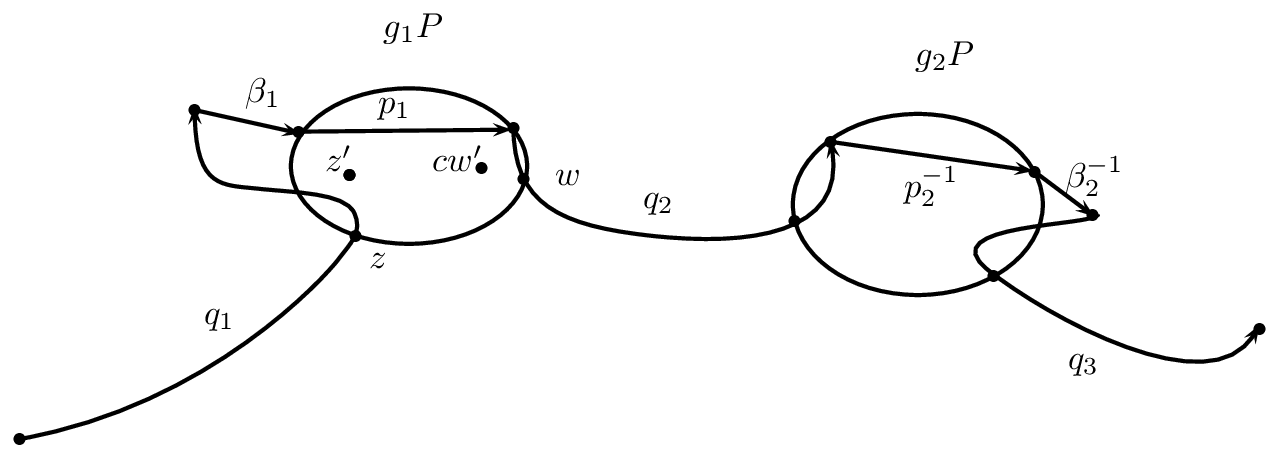} 
} \caption{Proof of Lemma \ref{hnnadm}} \label{fig:fig2}
\end{figure}

\subsection{Proof of Theorem \ref{hnn}}
Let $D = D(\lambda, (\lambda + 2)\card{f}), \Lambda=\Lambda(\lambda,
(\lambda + 2)\card{f})$ be the constants given by Corollary
\ref{maincor}. Define $$D(H, P, f) = D + 1 - \kappa(H, f^{-1} P, M)
- \kappa(H, P, M).$$  By Lemma \ref{hnnadm}, the truncation path of
any element in $H \star_{Q^t = Q'}$ is a $(\Lambda,
0)$-quasigeodesic in $\Gx $. Then $H \star_{Q^t = Q'} \to \langle H,
t\rangle$ is injective.

We shall now show that $\langle H, t\rangle$ is relatively
quasiconvex. Let $\bar \gamma = q_1' p_1' \ldots q_n' p_n'$ be a
truncation path obtained as above for an element in $\langle H,
t\rangle \setminus H$. Let $U = U(\Lambda, 0)$ be the constant given
by Lemma \ref{transpoints} and $L = \nu(U) + 1$. Observe that any
$(U, L)$-transition point of $\bar \gamma$ is a $(U, L)$-transition
point of either $q_i'$ or $p_i'$.

By the definition of truncation, each $q_i'$ is the lift of one of
the following types of relative paths: 1). $\beta_{i-1}^{-1} q_i
\beta_i$, 2). a subpath of either $\beta_{i-1}^{-1} q_i$, $q_i
\beta_i$ or  $q_i$. Note that $q_i$ has two endpoints in a left
$H$-coset and $\beta_i$ is of the fixed length $\card{f}$. Then any
$(U, L)$-transition point of $q_i'$ lies in a uniform neighborhood
of the associated left $H$-coset in all cases.

We now consider the subpath $p_i'$. By the analysis in the proof of
Lemma \ref{hnnadm} and the inequality (\ref{truncation}) therein, we
see that the endpoints of $p_i'$ are at most a distance $\kappa(H,
f^{-1}P, M) - \kappa(H, P, M)$ to the endpoints of a geodesic $[z',
cw']$ with label $z'^{-1} c w'  \in Q c Q = Q c$. Since $Q \subset
H$ and $z' \in Q$, any $(U, L)$-transition point of $[z', cw']$ lies
in a uniform neighborhood of the associated left $Q$-coset. Note
that $c$ is a fixed element. Then any $(U, L)$-transition point of
$p_i'$ has a uniformly bounded distance to a $(U, L)$-transition
point of $[z', cw']$. Consequently, any $(U, L)$-transition point of
$p_i'$ lies in a uniform neighborhood of the associated left
$H$-coset.

Therefore, it is verified that any $(U, L)$-transition point of
$\bar \gamma$ lies in a uniform neighborhood of $\langle H, t
\rangle$. This shows that the relative quasiconvexity of $\langle H,
t\rangle$.

We now show the second statement of Theorem \ref{hnn}.
\begin{lem}
Every parabolic subgroup in $\langle H, t \rangle$ is conjugate into
$H$.
\end{lem}
\begin{proof}[Sketch of Proof]
Let $g \in \langle H, t \rangle \setminus H$. Similarly as Lemma
\ref{amalgparabolic}, the idea is to take sufficiently large $D$ in
Theorem \ref{hnn}, to show that $g \notin P^f$ for any $f \in G$ and
$P \in \mathbb P$.

Suppose, to the contrary,  that $g = fpf^{-1}$, where $p \in P$. Let
$\bar \gamma = q_1 p_1 q_2 p_2 \ldots q_n p_n$ be the truncation
path of $g$, where $p_i$ are $P$-components.

Let $\alpha$ be a geodesic segment with the same endpoints as
$\gamma$. By Proposition \ref{admissible}, the endpoints of each
$p_i$ lie in a uniform $R$-neighborhood of $\alpha$, where $R =
R(1,0)$.

Set $U :=\sigma(R+\sigma(\mu_{1, 0}))$ as in the proof of Lemma
\ref{amalgparabolic}. By taking a sufficiently large power of $g$,
we can assume further that there exist $p_{i-1}, p_i$ of $\bar
\gamma$ such that $p_{i-1}, p_i \subset N_U(fP)$.

Note that each $p_i$ is a $P$-component with endpoints in some
parabolic coset $g_i P$. By the $\nu$-bounded intersection of
$\mathbb X$ and $D > \nu(U)$, we obtain that $g_i P = fP$. However,
$g_{i-1} P, g_i P$ are distinct by Lemma \ref{distneighbor}. This
gives a contradiction. Hence, it is shown that $g \notin P^f$ for
any $f \in G, P \in \mathbb P$.
\end{proof}

\subsection{Proof of Corollary \ref{surface}}
Let $\{Q_1, \dots, Q_m\}$ be the conjugacy classes in $H$
representing boundary components of a compact surface $S$. As $H$
has no accidental parabolics in $G$, there exists parabolic
subgroups $P_1, \ldots, P_m$ of $G$ such that $H$ is relatively
quasiconvex in $G$ and $Q_i = H \cap P_i$ are parabolic subgroups in
$H$.

By Theorem 1.3 in \cite{MarPed2}, there exists a constant $D_1=D(H,
P_1)$ such that the following holds. Let $pQ_1$ be such that
$\forall g \in pQ_1, d(1, g)> D_1$. Then we have that  $H
\star_{Q_1} H^{p} = \langle H, H^{p}\rangle$ is relatively
quasiconvex in $G$.

Note that $Q_i$ are all cyclic and $P_i$ are of rank at least two.
Then there exists $p_1 \in P_1$ such that any elements in $p_1 Q_1$
has length bigger then $D_1$. This implies that $H_1^+ = H
\star_{Q_1} H^{p_1} = \langle H, H^{p_1}\rangle$ is relatively
quasiconvex. Moreover, the complete set of conjugacy classes of
parabolic subgroups in $H_1^+$ is $\{Q_1, Q_2, Q_2^{p_1}, \ldots,
Q_m, Q_m^{p_1}\}$.

Apply Theorem \ref{hnn} to $H_1^+, P_2$ and $p_1$. Let $D_2 =
D(H_1^+, P_2, p_1)$ be the constant given by Theorem \ref{hnn}. As
$Q_2$ is of infinite index in $P_2$, there is an element $p_2 \in
P_2$ such that any element in $p_2 Q_2$ is of length bigger then
$D_2$. Let $t = p_1 p_2$. It follows that $H_2^+ =\langle H_1^+, t
\rangle$ is relatively quasiconvex. Moreover, the complete set of
conjugacy classes of parabolic subgroups in $H_2^+$ are $\{Q_1, Q_2,
\ldots, Q_m, Q_m^{p_m}\}$.

After a finitely many steps, we obtain that $H_m^+$ is relatively
quasiconvex and its conjugacy classes of parabolic subgroups are
$\{Q_1, Q_2, \ldots, Q_m\}$. On the other hand, one sees that
$H_m^+$ is isomorphic a closed surface group. The proof is complete.

\begin{rem}
Note that the constructed surface subgroup $H_m^+$ has accidental
parabolics in $G$.
\end{rem}

\section{Separability of double cosets}

Suppose $G$ is hyperbolic relative to a collection of slender LERF
groups. Note that every subgroup of a slender LERF group is
separable.

The following result is shown in \cite{ManMar} by a much involved
proof. We here provide simpler proof using normal paths constructed
in Section 3.
\begin{lem}\label{sepafull}
Suppose $G$ is hyperbolic relative to a collection of slender LERF
groups. Let $H$ be relatively quasiconvex in $G$ and an element $g
\in G \setminus H$. Then there exists a fully quasiconvex subgroup
$H^+$ such that $H \subset H^+$ and $g \notin H^+$.
\end{lem}
\begin{proof}
Let $g \in G \setminus H$. Let $P \in \mathbb P^G$ such that
$\card{H \cap P} =\infty$ and $[P : H \cap P] < \infty$. Since $H
\cap P$ is separable in $P$, we can combine $H$ with a finite index
subgroup $\dot P \subset P$, where each element in $\dot P \setminus
H$ has sufficiently large word length. Thus by Lemma \ref{inject1},
so does any element in $\langle H, \dot P \rangle \setminus (H \cap
P)$. As a consequence, this implies that $g \notin \langle H, \dot P
\rangle$ (otherwise, it would lead to a contradiction that $g \in H
\cap P$). After finitely many steps, we can get a fully quasiconvex
subgroup containing $H$ but avoiding $g$.
\end{proof}

\begin{cor}\label{relsepa}
Under the assumption of Lemma \ref{sepafull}, if $G$ is separable on
fully quasiconvex subgroups, then every relatively quasiconvex
subgroup is separable.
\end{cor}

Suppose $H$ is relatively quasiconvex and $K$ fully quasiconvex in
$G$. Let $C = H\cap K$.

\begin{proof}[Proof of Theorem \ref{doublecoset}]
Since $H', K'$ are separable, $C'=H'\cap K'$ is separable. Given any
$g \notin H'K'$, it suffices to find a closed set separating $H'K'$
and $g$.

We first consider the case that $C' = C$.

Let $m=d(1,g)$. Let $\Lambda = \Lambda(H, K), D = D(H, K)$ be the
constants given by Lemma \ref{inject1}, and $D_1 = \max(m\Lambda,
D)$. By the separability of $C$, there exists finite index subgroups
$\dot H, \dot K$ of $H', K'$ respectively such that $d(1, f) > D_1$
for any $f \in \dot H \cup \dot K \setminus C$. By Theorem
\ref{freeproduct}, $\langle \dot H, \dot K \rangle$ is relatively
quasiconvex. Thus $\langle \dot H, \dot K \rangle$ is separable by
Corollary \ref{relsepa}.

We now claim that $g \notin H'\langle \dot H , \dot K \rangle K'$.
Argue by way of contradiction. Suppose that $g \in H' \langle \dot H
, \dot K \rangle K'$. Then there exists $h \in H', k \in K'$ such
that $h g k \in \langle \dot H, \dot K \rangle$. Since $g \notin
H'K'$, it follows that $h g k$ cannot be written as $h' k'$, where
$h' \in \dot H, k' \in \dot K$. Hence, $g$ has the following form $g
= h_0 k_1 h_1 \ldots k_n h_n k_{n+1}$ for $n \ge 1$, where $h_0 \in
H', k_{n+1} \in K', h_i \in \dot H\setminus C, k_i \in \dot
K\setminus C$. Let $\gamma$ be the normal path of $g$. Since
$\gamma$ is a $(\Lambda, 0)$-quasigeodesic, we have $d(1, g) >
\Lambda^{-1} D_1 > m$. This is a contradiction. It follows that $g
\notin H'\langle \dot H , \dot K \rangle K'$.

Since $\dot H, \dot K$ are of finite index in $H', K'$ respectively,
there exists finitely many $h_i \in H', k_j \in K'$ such that $\cup
h_i \dot H = H', \cup \dot K' k_j = K$. Observe that $H'K'$ is
contained in  $H'\langle \dot H , \dot K \rangle K' = \bigsqcup h_i
\langle \dot H,  \dot K \rangle k_j $, which is a finite union of
closed sets. This shows that $'HK'$ is separable.

Let's now turn to the general case that $C'$ is of finite index in
$C$.

Denote by $\{1, c_1, c_2, ..., c_n \}$ the set of left coset
representatives of $C'$ in $C$. In virtue of separability of $H',
K'$, it is easy to see that there are two finite index subgroups
$\hat H, \hat K$ in $H, K$ respectively such that $\hat H \cap  \hat
K = C'$. Note that $\hat K$ fully quasiconvex. Applying the special
case to $H_1 = H' \cap \hat H, K_1 = K' \cap \hat K$, we see
$H_1K_1$ is separable. Since $H_1, K_1$ are of finite index in $H',
K'$ respectively, it follows that $H'K'$ is separable.
\end{proof}

\begin{proof}[Proof of Corollary \ref{dblparacoset}]
Let $H, K$ be two parabolic subgroups. If $H, K$ lie in different
maximal parabolic subgroups, then $H\cap K$ is finite. The
conclusion follows from Theorem \ref{doublecoset}. Now let $H, K$ be
in the same $P \in \mathbb P$. Since $P$ is virtually abelian, it
follows that the double coset of any two subgroups in $P$ is
separable in $P$ and thus in $G$.

Note that by a result of Osin \cite{Osin2}, any hyperbolic element
$g$ in $G$ is contained in a virtually cyclic subgroup $E(g)$ such
that $G$ is hyperbolic relative to $\mathbb P \cup \{E(g)\}$. Hence
it follows by the same argument as above that the double coset of
any two cyclic subgroups is separable.
\end{proof}








\bibliographystyle{amsplain}
\bibliography{bibliography}

\end{document}